\documentclass{article}

 \usepackage[utf8]{inputenc}
 \usepackage{scalerel}
 \usepackage{enumerate}
 \usepackage{authblk}
 \usepackage{wrapfig}
\usepackage{tikz-cd}
\usepackage{amsmath,amscd}
\usepackage{mathtools}
 \usepackage{amsthm}
 \usepackage{amsfonts}
 \usepackage{color}
\usepackage[mathscr]{eucal}
 \usepackage{amssymb}
 \usepackage[margin=1in]{geometry}
\usepackage{comment}
\usepackage[]{lineno}

\usepackage{hyperref}
\hypersetup{
    colorlinks,
    citecolor=black,
    filecolor=black,
    linkcolor=black,
    urlcolor=black
}

\graphicspath{{figures/}}

\makeatletter
\newtheorem*{rep@theorem}{\rep@title}
\newcommand{\newreptheorem}[2]{%
\newenvironment{rep#1}[1]{%
\def\rep@title{#2 \ref{##1}}%
\begin{rep@theorem}}%
{\end{rep@theorem}}}
\makeatother

 \newtheorem{theorem}{Theorem}[section]
 \newtheorem{proposition}[theorem]{Proposition}
 \newtheorem{corollary}[theorem]{Corollary}
 \newtheorem{lemma}[theorem]{Lemma}

 \newreptheorem{theorem}{Theorem}
 
 \theoremstyle{definition}
 \newtheorem{definition}[theorem]{Definition}
 
  \newtheorem{remark}[theorem]{Remark}

\usepackage{chngcntr}
\usepackage{apptools}
\AtAppendix{\counterwithin{theorem}{section}}

\definecolor{darkblue}{rgb}{0.0, 0.0, 0.8}
\definecolor{darkred}{rgb}{0.8, 0.0, 0.0}
\definecolor{darkgreen}{rgb}{0.0, 0.8, 0.0}
\definecolor{ncolor}{rgb}{0.8, 0.8, 0.0}

\newcommand{\supp}{\mathrm{supp}}

\newcommand{\fmm}[1]{\mathcal{#1}}
\newcommand{\mm}[1]{\mathscr{#1}}
\newcommand{\dgh}{d_{GH}}
\newcommand{\dgp}{d_{GP}}

\newcommand{\dha}{d_{H}}
\newcommand{\diam}{\mathrm{diam}}
\newcommand{\rad}{\mathrm{rad}}
\newcommand{\dis}{\mathrm{dis}}
\newcommand{\dist}{\mathrm{dist}}

\newcommand{\real}{\mathbb{R}}
\newcommand{\R}{\real}

\newcommand{\vr}{V\!R}

\newcommand{\dwa}[1]{d_{W,{#1}}}
\newcommand{\dwp}{d_{W,p}}
\newcommand{\dwi}{d_{W,\infty}}
\newcommand{\dgw}[1]{d_{GW,{#1}}}
\newcommand{\dgwp}{d_{{GW,p}}}
\newcommand{\dgwi}{\ensuremath{d_{GW,\infty}}}
\newcommand{\dto}{\ensuremath{\Rightarrow}}
\newcommand{\dI}{\ensuremath{d_I}}
\newcommand{\dHI}{d_{HI}}
\newcommand{\dpr}{\ensuremath{d_P}}

\newcommand{\calX}{\mathcal{X}}
\newcommand{\calY}{\mathcal{Y}}
\newcommand{\calZ}{\mathcal{Z}}

\newcommand{\calMM}{\widehat{\mathcal{F}}}
\newcommand{\calC}{\mathcal{C}}
\newcommand{\calE}{\mathcal{E}}
\newcommand{\calF}{\mathcal{F}}
\newcommand{\calR}{\mathcal{R}}

\newcommand{\calRB}{\mathcal{R}_B}
\newcommand{\calRBn}{\mathcal{R}^n_B}
\newcommand{\calDX}{\mathcal{D}_\mathcal{X}}
\newcommand{\calDY}{\mathcal{D}_\mathcal{Y}}
\newcommand{\calDXn}{\mathcal{D}_\mathcal{X}^n}
\newcommand{\calDYn}{\mathcal{D}_\mathcal{Y}^n}
\newcommand{\calFX}{\mathcal{F}_{\mathcal{X}}}
\newcommand{\calFY}{\mathcal{F}_{\mathcal{Y}}}
\newcommand{\calFXn}{\mathcal{F}_{\mathcal{X}}^n}
\newcommand{\calFYn}{\mathcal{F}_{\mathcal{Y}}^n}

\newcommand{\muX}{\mu_X}
\newcommand{\muY}{\mu_Y}

\newcommand{\psiX}{\psi_X}
\newcommand{\psiY}{\psi_Y}
\newcommand{\iX}{\iota_X}
\newcommand{\iY}{\iota_Y}
\newcommand{\piX}{\pi_X}
\newcommand{\piY}{\pi_Y}
\newcommand{\piZ}{\pi_Z}
\newcommand{\dX}{d_X}
\newcommand{\dY}{d_Y}
\newcommand{\dZ}{d_Z}
\newcommand{\dB}{d_B}
\newcommand{\mXY}{m_{X,Y}}
\newcommand{\cXY}{C(\muX,\muY)}
\newcommand{\dXY}{d_{X,Y}}

\newcommand{\dP}{d\mu}
\newcommand{\dPP}{d(\mu \otimes \mu)}
\newcommand{\dPn}{d\mu_n}
\newcommand{\dPPn}{d(\mu_n \otimes \mu_n)}
\newcommand{\sP}{\mathrm{supp} \,\mu}
\newcommand{\sPP}{\mathrm{supp} \,(\mu \otimes \mu)}

\newcommand{\costpPm}{\bigg( \int \mXY^p \,\dPP \bigg)^{1/p}}
\newcommand{\costpPd}{\bigg( \int \dXY^p \,\dP \bigg)^{1/p}}
\newcommand{\costiPm}{\sup_{\sPP}  \mXY}
\newcommand{\costiPd}{\sup_{\sP}\, \dXY}
\newcommand{\costiP}{\max \left\{ \frac{1}{2} \costiPm,
\, \costiPd  \right\}}
\newcommand{\costpP}{ \max \bigg\{ \frac{1}{2} \costpPm, \costpPd \bigg\}}

\newcommand{\unxpe}{U_{X,p}^{n,\epsilon}}
\newcommand{\unype}{U_{Y,p}^{n,\epsilon}}
\newcommand{\unxqe}{U_{X,q}^{n,\epsilon}}
\newcommand{\unyqe}{U_{Y,q}^{n,\epsilon}}
\newcommand{\cnxqe}{C_{X,q}^{n,\epsilon}}
\newcommand{\cnyqe}{C_{Y,q}^{n,\epsilon}}

\newcommand{\cat}[1]{\text{#1}}

\newcommand{\urysohnB}{\mathscr{U}_B}

\newcommand{\lawsB}{\mathcal{L}_B}

\begin{document}

\title{On Metrics for Analysis of Functional Data on Geometric Domains\footnote{An early version of this manuscript was posted to arXiv as \cite{anbouhi2022universal}.}}
\author[1]{Soheil Anbouhi}
\author[2]{Washington Mio}
\author[3]{Osman Berat Okutan}
\affil[1]{Department of Mathematics and Statistics, Haverford College\\ \texttt{sanbouhi@haverford.edu}}
\affil[2]{Department of Mathematics, Florida Sate University\\ \texttt{wmio@fsu.edu}}
\affil[3]{Max Planck Institute for Mathematics in the Sciences\\ \texttt{okutan@mis.mpg.de}}
\date{ }

\maketitle

\vspace{-0.2in}
\begin{abstract}
This paper employs techniques from metric geometry and optimal transport theory to address questions related to the analysis of functional data on metric or metric-measure spaces, which we refer to as fields. Formally, fields are viewed as 1-Lipschitz mappings between Polish metric spaces with the domain possibly equipped with a Borel probability measure. We introduce field analogues of the Gromov-Hausdorff, Gromov-Prokhorov, and Gromov-Wasserstein distances, investigate their main properties and provide a characterization of the Gromov-Hausdorff distance in terms of isometric embeddings in a Urysohn universal field. Adapting the notion of distance matrices to fields, we formulate a discrete model, obtain an empirical estimation result that provides a theoretical basis for its use in functional data analysis, and prove a field analogue of Gromov's Reconstruction Theorem. We also investigate field versions of the Vietoris-Rips and neighborhood (or offset) filtrations and prove that they are stable with respect to appropriate metrics.
\end{abstract}

\smallskip

{\em Keywords:} functional data, Urysohn field, optimal transport, functional curvature.

\smallskip
{\em 2020 Mathematics Subject Classification:} 51F30, 60B05, 60B10.

\tableofcontents


\section{Introduction}\label{sec:intro}

\subsection{Overview}

This paper addresses problems in {\em functional metric geometry} that arise in the study of data such as signals recorded on geometric domains or the nodes of a network. Formally, these may be viewed as functions defined on metric spaces, sometimes equipped with additional structure such as a probability measure, in which case the domain is referred to as a {\em metric-measure space}, or simply $mm$-space. Datasets comprising such objects arise in many domains of scientific and practical interest. For example, on a social network, the edges normally represent some form of direct interaction but a metric $d$ on $V$, such as the shortest-path distance, the diffusion distance, or the commute-time distance \cite{coifman, martinez2019probing}, is useful in quantifying indirect interactions as well. A probability distribution $\mu$ on its set $V$ of nodes can be used to describe how influential the various members of the network are. The triple $(V,d,\mu)$ defines an $mm$-space. Attributes such as individual preferences, traits or characteristics may be viewed as a function $f \colon V \to B$, where $B$ is a metric space such as $\real^n$ for vector-valued attributes, or $\mathbb{Z}_2^n$ (binary words of length $n$) equipped with the $\ell_1$-norm for discrete attributes such as a like-or-dislike compilation of preferences. The quadruple $(V,d,\mu,f)$ is a functional $mm$-space that can be employed for data representation and analysis in many different scenarios. Social networks are dynamic, with individuals joining and leaving the network, their relevance changing over time, as well as their attributes \cite{tantipathananandh2007framework, sekara2016fundamental,doreian2013evolution}. This leads to a family of functional $mm$-space $(V_t, d_t, \mu_t, f_t)$ parameterized by time. To analyze, visualize and summarize dynamical structural and functional changes, it is important to define metrics that are sensitive to such changes and amenable to computation. Targeting problems such as this involving functional data, our primary goal is threefold: (i) to develop metrics that allow us to model and quantify variation in functional data, possibly with distinct domains, (ii) to investigate principled empirical estimations of these metrics, and (iii) to construct stable filtered spaces or complexes associated with functional data to enable geometric analysis via topological methods.

The analysis of structural variation in datasets comprising geometric objects has been a subject of extensive study using techniques from areas such as metric geometry, optimal transport theory, and topological data analysis. The Gromov-Hausdorff distance $d_{GH}$ \cite{burago2001course} has played a prominent role in quantifying shape contrasts and similarities in families of compact metric spaces and in addressing the stability of topological signatures of the shape of objects such as point clouds \cite{chazal2009gromov, carlsson2010multiparameter}. These topological signatures also provide lower bounds to $d_{GH}$ that are computationally more accessible \cite{chazal2009gromov}. For $mm$-spaces, a similar role is played by the Gromov-Prokhorov distance $d_{GP}$ \cite{gromov2007metric, greven2009convergence, loehr2013equivalence, janson2020gromov} and the Gromov-Wasserstein distance $d_{GW}$ \cite{memoli2011gromov, sturm2006geometry} that highlight the shape of regions of larger probability mass. As ubiquitous as $d_{GH}$, $d_{GP}$ and $d_{GW}$ have been in geometric and topological data analysis \cite{chazal2009gromov, memoli2011gromov, blumberg2022stability}, only a few aspects of their functional counterparts have been investigated (cf.\,\cite{bauer2014measuring, chazal2009gromov, carlsson2010multiparameter, hang2019topological, vayer2020fused}) with a  more thorough study seemingly lacking in the literature. This paper carries out such a study and also investigates discrete representations of $mm$-fields by means of {\em functional curvature sets} that encode their structural and functional shape.

The aforementioned topological signatures are frequently derived from filtered complexes or spaces such as the Vietoris-Rips filtration of a metric point cloud \cite{vietoris1927hoheren, carlsson2006algebraic}, the neighborhood filtration (also known as the offset filtration) of a compact subspace of a metric space \cite{chazal2009gromov, halperin2015offset}, or the metric-measure bifiltration of an $mm$-space \cite{blumberg2022stability, sheehy2012multicover}. For this reason, we also study variants of such (multiparameter) filtrations in the functional setting, establishing stability results that ensure that they can be used reliably in data analysis.

\subsection{Main Results}

We study the class of 1-Lipschitz maps $\pi_X \colon X \to B$, where $X$ and $B$ are Polish (complete and separable) metric spaces, $B$ fixed. These mappings are the morphisms in the category $\cat{Met}$ of metric spaces restricted to Polish spaces. We refer to such 1-Lipschitz mappings as {\em $B$-valued fields} on $X$, or simply $B$-fields, and denote them as triples $\mm{X}=(X, d_X, \pi_X)$. If $(X,d_X)$ is also equipped with a Borel probability measure $\mu_X$, we refer to the quadruple $\mm{X}=(X, d_X, \pi_X, \mu_X)$ as an $mm$-field over $B$.

\medskip 
\noindent
{\bf Metric Fields.} We define the Gromov-Hausdorff distance $\dgh$ between compact $B$-fields as the infimum of the Hausdorff distance between isometric embeddings into common $B$-fields. Analogous to the corresponding result for compact metric spaces \cite{burago2001course}, we provide a characterizations for $\dgh$ in terms of functional distortions and use it to prove the following theorem.
\begin{reptheorem}{thm:gh_metric}
The Gromov-Hausdorff distance $\dgh$ metrizes the moduli space $\fmm{F}_B$ of compact $B$-fields and $(\fmm{F}_B,\dgh)$ is a Polish metric space.
\end{reptheorem}
We also show in Proposition \ref{prop:geo} that $(\fmm{F}_B,\dgh)$ is a geodesic space if and only if $B$ is a geodesic space. A second characterization of $\dgh$ is in terms of isometric embeddings into a fixed {\em Urysohn universal field} $\mm{U}_B$ modulo the action of isometries. This Urysohn $B$-field has the property that any other $B$-field can be isometrically embedded in it and any two embeddings of the same $B$-field differ by the (left) action of an automorphism of $\mm{U}_B$. 

Let $F(\mm{U}_B)$ be the space of compact subfields of a Urysohn field $\mm{U}_B = (U,B,\pi_U)$, equipped with the Hausdorff distance, and $Aut(B)$ the automorphism group of $\mm{U}_B$, which acts on $\mm{U}_B$ by isometries. Denote the quotient metric on $F(\mm{U}_B)/Aut(B)$ by $d_F^B$.
\begin{reptheorem}{thm: F_B = F_I}
The moduli space $(\fmm{F}_B, \dgh)$ of isometry classes of compact $B$-fields equipped with the Gromov-Hausdorff distance is isometric to the quotient space $(F(\mm{U}_B)/Aut(B), d_F^B)$.
\end{reptheorem}

\medskip
\noindent
{\bf Metric-Measure Fields.} We define and investigate the main properties of $mm$-field analogues of the Gromov-Prokhorov and Gromov-Wasserstein distances that have been studied extensively in the realm of $mm$-spaces \cite{gromov2007metric, greven2009convergence,memoli2011gromov, sturm2006geometry}. For $mm$-fields $\mathcal{X}$ and $\mathcal{Y}$ over $B$, the Gromov-Prokhorov distance is denoted $\dgp(\mathcal{X},\mathcal{Y})$, whereas the Gromov-Wasserstein distance depends on a parameter $1 \leq p \leq \infty$ and is denoted $\dgw{p}(\mathcal{X},\mathcal{Y})$. Two different approaches to Gromov-Wasserstein have been developed for $mm$-spaces in \cite{memoli2011gromov} and \cite{sturm2006geometry} and the field counterpart we present is along the lines of \cite{memoli2011gromov}.

We define $\dgp(\mathcal{X},\mathcal{Y})$ as the infimum of the Hausdorff distances between isometric embeddings of $\mathcal{X}$ and $\mathcal{Y}$ into commom $B$-fields and Theorem \ref{thm:gp_coupling} shows how to express $\dgp$ in terms of couplings. This characterization is used in Theorem \ref{thm:gp_metric} to prove that $\dgp$ metrizes the set $\calMM_B$ of isometry classes of fully supported $mm$-fields over $B$ making $(\calMM_B, \dgp)$ a Polish space.

The Gromov-Wasserstein distance $\dgw{p} (\mathcal{X},\mathcal{Y})$, $1 \leq p \leq \infty$, is defined through couplings and Theorem \ref{thm:gwuniform} provides the following characterization of $\dgw{\infty}$ in terms of equidistributed sequences $(x_i)$ and $(y_i)$ in $(X,d_X, \mu_X)$ and $(Y, d_Y, \mu_Y)$, respectively:
\[
\dgwi(\calX,\calY)= \inf \max \left\{ \frac{1}{2} \sup_{i,j} |d_X(x_i,x_j) - d_Y (y_i,y_j)|, \, \sup_i d_B(\pi_X(x_i), \pi_Y (y_i)) \right\},
\]
with the infimum taken over all equidistributed sequences $(x_i)$ and $(y_i)$. (A sequence $(x_i)$ in $(X, d_X)$ is $\mu_X$-equidistributed if the empirical measures $\sum_{i=1}^n \delta_{x_i}/n$ converge weakly to $\mu_X$.) To our knowledge, this result is new even for the Gromov-Wasserstein distance between metric-measure spaces, a result that follows by taking the target space $B$ to be a singleton.

With an eye toward empirical estimation of the Gromov-Wasserstein distance between $mm$-fields, we introduce the notion of {\em extended distance matrices}, much in the way distance matrices are used to study $mm$-spaces (cf.\,\cite{gromov2007metric, gomez2021curvature}). For an $mm$-field $\fmm{X}$ over $B$ and a sequence $(x_i)$ in $X$, $i \geq 1$, form the countably infinite (pseudo) distance matrix $R = (r_{ij}) \in \real^{\mathbb{N} \times \mathbb{N}}$ and the infinite sequence $b = (b_i) \in B^{\mathbb{N}}$, where $r_{ij} = d_X (x_i, x_j)$ and $b_i = \pi(x_i) \in B$. We refer to the pair $(R, b)$ as the {\em augmented distance matrix} of $\fmm{X}$ associated with the sequence $(x_i)$, which records the shape of the graph of $\pi_X$ restricted to the sequence. This construction let us define a Borel measurable mapping $F_{\fmm{X}} \colon X^\infty \to \real^{\mathbb{N} \times \mathbb{N}} \times B^\infty$, with both the domain and co-domain equipped with the weak topology. The pushforward of $\mu^\infty$ under $F_{\fmm{X}}$ yields a probability measure on $\real^{\mathbb{N} \times \mathbb{N}} \times B^\infty$ that we denote by
\begin{equation}
\fmm{D}_{\fmm{X}}:= (F_{\fmm{X}})_\ast (\mu^\infty) \,,
\end{equation}
and refer to as the {\em field curvature distribution} of $\fmm{X}$. This terminology is motivated by the notion of {\em curvature set} of a metric space \cite{gromov2007metric}. A similar construction for finite sequences $\{x_i\}$, $1 \leq i \leq n$, gives a measure $\fmm{D}^n_{\fmm{X}}$ on $\real^{n \times n} \times B^n$.

Our main result supporting empirical estimation of the Gromov-Wasserstein distance between $mm$-fields is the following convergence theorem.
\begin{reptheorem}{thm: dWp(Dx,Dy)=dGW(infty)}
Let $\calX$ and $\calY$ be bounded $mm$-fields over $B$. Then, for any $1 \leq p \leq \infty$, we have
\begin{equation*}
    \lim_{n \to \infty} \dwp(\calDXn,\calDYn) = \dwp(\mathcal{D}_\mathcal{X},\mathcal{D}_\mathcal{Y}) = \dgwi(\mathcal{X},\mathcal{Y}) .
\end{equation*}
\end{reptheorem}

A metric-measure field counterpart to Gromov's Reconstruction Theorem for $mm$-spaces \cite{gromov2007metric} is a special case of this more general result. The $mm$-field reconstruction theorem is stated and proven in Theorem \ref{thm:reconstruction}.

\subsection{Organization}

Section \ref{sec:gh} introduces the Gromov-Hausdorff distance between compact fields and provides a characterization of $d_{GH}$ in terms of distortions of (functional) correspondences, whereas Section \ref{sec:urysohn} shows that $d_{GH}$ can be realized as a Hausdorff distance through isometric embeddings in a Urysohn universal field. The $d_{GP}$ and $d_{GW}$ distances for metric-measure fields are studied in Sections \ref{sec:gp} and \ref{sec:gw}, respectively. Section \ref{sec:dmatrix} introduces a representation of $mm$-fields by distributions of infinite augmented distance matrices, proves a field reconstruction theorem based on these distributions and also addresses empirical approximation questions. Section \ref{sec:filtrations} introduces functional analogues of the Vietoris-Rips and neighborhood filtrations and proves their stability. Section \ref{sec:summary} closes the paper with a summary and some discussion.


\section{Gromov-Hausdorff Distance for Metric Fields}\label{sec:gh}

A  $B$-field is a 1-Lipschitz map $\pi \colon X \to B$, where $X$ and $B$ are Polish metric spaces. We sometimes denote the field as a triple $\mm{X}=(X, d_X, \pi)$.

\begin{definition} 
Let $\mm{X}=(X,d_X,\pi_X)$ and $\mm{Y}=(Y,d_Y,\pi_Y)$ be $B$-fields. 
\begin{enumerate}[(i)]
\item A mapping from $\mm{X}$ to $\mm{Y}$ over $B$, denoted $\Phi \colon \mm{X} \dto \mm{Y}$,  consists of a 1-Lipschitz mapping $\phi \colon X \to Y$ such that the diagram
\begin{equation}
\begin{tikzcd}
X \arrow[rd, "\pi_X"'] \arrow[rr,  "\phi"] & & Y \arrow[ld, "\pi_Y"] \\
& B &
\end{tikzcd}
\end{equation}
commutes. We adopt the convention that uppercase letters denote maps between fields and the corresponding lowercase letters denote the associated mapping between their domains.
\item $\Phi \colon \mm{X} \dto \mm{Y}$ is an {\em isometric embedding} if $\phi: X\to Y$ is an isometric embedding of metric spaces.
\item $\Phi \colon \mm{X} \dto \mm{Y}$ is an {\em isometry} if $\phi: X \to Y$ is a bijective isometry. Moreover, we say that $\mm{X}$ and $\mm{Y}$ are {\em isometric}, denoted by $\mm{X} \simeq \mm{Y}$, if there is an isometry $\Phi$ between them. If $\mm{X} = \mm{Y}$, we refer to an isometry as an {\em automorphism} of $\mm{X}$.
\end{enumerate}
\end{definition}

\begin{definition} (Gromov-Hausdorff Distance for Fields) \label{def: functional GHB distance}
Let $\mm{X}=(X,d_X,\pi_X)$ and $\mm{Y}=(Y,d_Y,\pi_Y)$ be $B$-fields. The {\em Gromov-Hausdorff distance} is defined by
\begin{equation*}
    \dgh(\mm{X}, \mm{Y}) := \inf_{Z,\Phi,\Psi} d^Z_H(\phi(X),\psi(Y)),
\end{equation*}
where the infimum is taken over all $B$-fields $\mm{Z}$ and isometric embeddings $\Phi \colon \mm{X} \dto \mm{Z}$ and $\Psi \colon \mm{Y} \dto \mm{Z}$ over $B$. 
\end{definition}

\begin{remark}
The definition of $\dgh$ only involves the Hausdorff distance between the domains of the functions once embedded in $\mm{Z}$. This is due to the fact that, since fields are 1-Lipschitz, differences in function values are bounded above by distances in the domain. 
\end{remark}

Definition \ref{def: functional GHB distance} provides a field version of the Gromov-Hausdorff distance for metric spaces \cite{burago2001course}. In the remainder of this section, we study some of the basic properties of $\dgh$. Most arguments, through the proof of Theorem \ref{thm:gh_metric}, are similar to those in \cite{burago2001course} for metric spaces but are included here for completness.

The Gromov-Hausdorff distance is non-negative, symmetric and a field has zero distance to itself. To show that the Gromov-Hausdorff distance is a pseudo-metric over the class of fields, we need to show the triangle inequality. To this end, we first prove a gluing lemma for $B$-fields.

Let $\mm{Y},\mm{Z}_1,\mm{Z}_2$ be $B$-fields and $\Phi \colon \mm{Y} \hookrightarrow \mm{Z}_1$, $\Psi \colon \mm{Y} \hookrightarrow \mm{Z}_2$ isometric embeddings. As a set, the amalgamation of $Z_1$ and $Z_2$ along $Y$, denoted $Z=Z_1 \coprod_Y Z_2$, is the quotient of the disjoint union $Z_1 \coprod Z_2$ under the equivalence relation generated by the relations $z_1 \sim z_2$ if $z_1 \in Z_1$, $z_2 \in Z_2$, and there exists $y \in Y$ such that $\phi(y)=z_1$ and $\psi(y)=z_2$. Define $d_Z \colon Z \times Z \to [0,\infty)$ by $d_Z|_{Z_1 \times Z_1}:=d_{Z_1}$, $d_Z|_{Z_2 \times Z_2}:=d_{Z_2}$ and
\begin{equation}
d_Z(z_1,z_2):=\inf_{y \in Y} d_{Z_1}(z_1,\phi(y))+d_{Z_2}(\psi(y),z_2),
\end{equation}
for $z_1 \in Z_1$ and $z_2 \in Z_2$, and $\pi_Z \colon Z \to B$ by $\pi_Z|_{Z_1}:=\pi_{Z_1}$ and $\pi_Z|_{Z_2}:=\pi_{Z_2}$.

\begin{lemma}[Gluing Lemma]\label{lem:gluing_fields}
    If $\mm{Y},\mm{Z}_1,\mm{Z}_2$ are $B$-fields and $\Phi \colon \mm{Y} \hookrightarrow \mm{Z}_1$, $\Psi \colon \mm{Y} \hookrightarrow \mm{Z}_2$ are isometric embeddings, then $d_Z \colon Z \times Z \to [0,\infty)$ is a well-defined metric on $Z$ and $\pi_Z \colon Z \to B$ is 1-Lipschitz. Hence, $\mm{Z}=(Z,d_Z,\pi_Z)$ is a $B$-field, and $\mm{Z}_1,\mm{Z}_2$ are isometrically included in $\mm{Z}$.
\end{lemma}
\begin{proof}
    To show that $d_Z$ is well defined, we verify that $d_Z(z_1,\psi(y))=d_Z(z_1,\phi(y))$ and $d_Z(z_2,\psi(y))=d_Z(z_2,\phi(y))$, for any $z_1 \in Z_1$, $z_2 \in Z_2$ and $y \in Y$. Indeed,
    \begin{equation}
        \begin{split}
            d_{Z}(z_1,\psi(y))&= \inf_{y' \in Y} d_{Z_1}(z_1,\phi(y'))+d_{Z_2}(\psi(y'),\psi(y)) = \inf_{y' \in Y} d_{Z_1}(z_1,\phi(y'))+d_Y(y',y) \\
                              &=\inf_{y' \in Y} d_{Z_1}(z_1,\phi(y'))+d_{Z_1}(\phi(y'),\phi(y)) = d_{Z_1}(z_1,\phi(y))=d_Z(z_1,\phi(y)).
        \end{split}
    \end{equation}
    Similarly, $d_Z(z_2,\psi(y))=d_Z(z_2,\phi(y))$. Thus, $d_Z$ is well defined. To show that $d_Z$ is a metric, we show that definiteness and the triangle inequality hold, since other properties of a metric follow easily from the definition. Assume that $z_1 \in Z_1$, $z_2 \in Z_2$ and $d_Z(z_1,z_2)=0$. For each integer $n>0$, there exists $y_n \in Y$ such that $d_{Z_1}(z_1,\phi(y_n)) \leq 1/n$, $d_{Z_2}(z_2,\psi(y_n)) \leq 1/n$. This shows that $(y_n)$ is a Cauchy sequence in $Y$. By completeness, it has a limit $y \in Y$. Then, $z_1=\phi(y)$ and $z_2=\psi(y)$, implying that $z_1$ is equal to $z_2$ in $Z$. This shows definiteness.

    For the triangle inequality, let $z_1,z_1' \in Z_1$, and $z_2 \in Z_2$. We have
    \begin{equation}
        \begin{split}
            d_Z(z_1,z_1')+d_Z(z_1',z_2) &= d_{Z_1}(z_1,z_1')+\inf_{y \in Y} d_{Z_1}(z_1',\phi(y))+d_{Z_2}(\psi(y),z_2) \\
                                        &\geq \inf_{y \in Y} d_{Z_1}(z_1,\phi(y))+d_{Z_2}(\psi(y),z_2)=d_Z(z_1,z_2)
        \end{split} 
    \end{equation}
   and
    \begin{equation}
        \begin{split}
            d_Z(z_1,z_2)+d_Z(z_2,z_1') &= \inf_{y,y' \in Y} d_{Z_1}(z_1,\phi(y))+d_{Z_1}(z_1',\phi(y'))+d_{Z_2}(z_2,\psi(y))+d_{Z_2}(z_2,\psi(y')) \\ 
            &\geq \inf_{y,y' \in Y} d_{Z_1}(z_1,\phi(y))+d_{Z_1}(z_1',\phi(y'))+d_{Z_2}(\psi(y),\psi(y')) \\
            &=\inf_{y,y' \in Y}d_{Z_1}(z_1,\phi(y))+d_{Z_1}(z_1',\phi(y'))+d_{Z_1}(\phi(y),\phi(y')) \\
            &\geq \ d_{Z_1}(z_1,z_1')=d_Z(z_1,z_1').
        \end{split}
    \end{equation}
    Similarly, for $z_1 \in Z_1$ and $z_2,z_2' \in Z_2$, we have
    \begin{equation}
    d_Z(z_1,z_2)+d_Z(z_2,z_2') \geq d_Z(z_1,z_2'),\, d_Z(z_2,z_1)+d_Z(z_1,z_2') \geq d_Z(z_2,z_2').
    \end{equation}
    Thus, the triangle inequality holds and $d_Z$ is a metric on $Z$, as claimed.

    The map $\pi_Z$ is well defined because if $z_1=\phi(y)$ and $z_2=\psi(y)$, then $\pi_{Z_1}(z_1)=\pi_{Z_2}(z_2)=\pi_Y(y)$. Let us show that $\pi_Z$ is $1$-Lipschitz. If  $z_1 \in Z_1$ and $z_2 \in Z_2$, then
    \begin{equation}
        \begin{split}
            d_Z(z_1,z_2)&=\inf_{y \in Y} d_{Z_1}(z_1,\phi(y))+d_{Z_2}(z_2,\psi(y)) \\
                        &\geq \inf_{y \in Y} d_B(\pi_{Z_1}(z_1),\pi_{Z_1}(\phi(y)))+d_B(\pi_{Z_2}(z_2),\pi_{Z_2}(\psi(y))) \\
                        &=\inf_{y \in Y} d_B(\pi_{Z_1}(z_1),\pi_Y(y))+d_B(\pi_{Z_2}(z_2),\pi_Y(y)) \\
                        &\geq d_B(\pi_{Z_1}(z_1),\pi_{Z_2}(z_2))=d_B(\pi_Z(z_1),\pi_Z(z_2)),
        \end{split}
    \end{equation}
    as desired. This completes the proof.
\end{proof}

\begin{proposition}\label{prop:gh_triangle}
    Let $\mm{X},\mm{Y},\mm{W}$ be $B$-fields. Then, $\dgh(\mm{X},\mm{W}) \le \dgh(\mm{X},\mm{Y})+\dgh(\mm{Y},\mm{W})$.
\end{proposition}
\begin{proof}
    Let $r>\dgh(\mm{X},\mm{Y})$ and $s>\dgh(\mm{Y},\mm{W})$. There exists $B$-fields $\mm{Z}_1$ and $\mm{Z}_2$, and isometric embeddings $\Lambda_1 \colon \mm{X} \hookrightarrow \mm{Z}_1$, $\Phi \colon \mm{Y} \hookrightarrow \mm{Z}_1$, $\Psi \colon \mm{Y} \hookrightarrow \mm{Z}_2$, and $\Lambda_2 \colon \mm{W} \hookrightarrow \mm{Z}_2$ such that $\dha^{Z_1}(\lambda_1 (X),\phi(Y)) < r$ and $\dha^{Z_2}(\psi(Y),\lambda_2 (W)) < s$. Let $\mm{Z}$ be the $B$-field described in Lemma \ref{lem:gluing_fields}. We show that $\dha^Z(\lambda_1 (X),\lambda_2(W))<r+s$. Given $x \in X$, there exist $y \in Y$ and $w \in W$ such that $d_{Z_1}(\lambda_1(x),\phi(y))<r$ and  $d_{Z_2}(\psi(y),\lambda_2 (w)) < s$. Then,
    \begin{equation}
    d_Z(\lambda_1 (x),\lambda_2(w)) \leq d_Z(\lambda_1(x),\phi(y))+d_Z(\phi(y),\psi(y))+d_Z(\psi(y),\lambda_2(w))<r+s.
    \end{equation}
    Similarly, for any $w$ in $W$, there exists $x$ in $X$ such that $d_Z(\lambda_1 (x),\lambda_2 (w))<r+s$. Hence,
    \begin{equation}
    \dgh(\mm{X},\mm{Y}) \leq \dha^Z(\lambda_1(X),\lambda_2(W))<r+s.
    \end{equation}
    Since $r>\dgh(\mm{X},\mm{Y})$ and $s>\dgh(\mm{Y},\mm{W})$ are arbitrary, we obtain $\dgh(\mm{X},\mm{W}) \leq \dgh(\mm{X},\mm{Y})+\dgh(\mm{Y},\mm{W})$ as desired.
\end{proof}

In analogy with the corresponding result for compact metric spaces, next, we obtain a characterization of the Gromov-Hausdorff distance for $B$-fields in terms of correspondences. Recall that a correspondence between two sets $X$ and $Y$ is a relation $R \subseteq X \times Y$ such that the projection to each factor restricted to $R$ is surjective.

\begin{definition}[Metric Field Distortion]
    Let $\mm{X}=(X,d_X,\pi_X)$ and $\mm{Y}=(Y,d_Y,\pi_Y)$ be  $B$-fields and $R$ be a relation between $X$ and $Y$. The \emph{metric field distortion} of $R$, denoted $\dis_{\pi_X,\pi_Y}(R)$, is defined as
    $$\dis_{\pi_X,\pi_Y}(R):=\max\big(\sup_{(x,y),(x',y') \in R} |d_X(x,x')-d_Y(y,y')|, 2\sup_{(x,y) \in R}d_B(\pi_X(x),\pi_Y(y)) \big).$$
\end{definition}

The following construction introduces a 1-parameter family of $B$-fields associated with a relation and its distortion. 

\begin{definition}\label{def:connecting_fields}
    Let $\mm{X}$ and $\mm{Y}$ be  $B$-fields, $R$ a relation between $X$ and $Y$, and $r>0$ satisfying $r \geq \dis_{\pi_X,\pi_Y}(R)/2$. Let $Z$ be the disjoint union of $X$ and $Y$, and $d_Z: Z \times Z \to \R$ be given by $d_Z|_{X \times X}=d_X$, $d_Z|_{Y \times Y}=d_Y$ and
    \[
    d_Z(x,y)=d_Z(y,x)=r+\inf_{(x',y') \in R} d_X(x,x')+d_Y(y,y'),
    \]
    for $x \in X$ and $y \in Y$. Letting $\pi_Z: Z \to B$ be given by $\pi_Z|_X=\pi_X$ and $\pi_Z|_Y=\pi_Y$, we define the $B$-field $\mm{X} \coprod_{R,r} \mm{Y}$ by $\mm{X} \coprod_{R,r} \mm{Y}:= (Z,d_Z,\pi_Z)$ (see next lemma).
\end{definition}

\begin{lemma}\label{lem:connect_by_relation}
    Let $\mm{X}$ and $\mm{Y}$ be  $B$-fields, $R$ a relation between $X$ and $Y$, and $r \geq \dis_{\pi_X,\pi_Y}(R)/2$. Then, $\mm{X} \coprod_{R,r} \mm{Y}$ is a $B$-field.
\end{lemma}
\begin{proof}
    Let $(Z,d_Z,\pi_Z):=\mm{X} \coprod_{R,r} \mm{Y}$. Definiteness and symmetry of $d_Z$ are clear from the definition. We verify the triangle inequality. Let $x_1,x_2 \in X$, and $y_1,y_2 \in Y$. We have
    \begin{equation}
        \begin{split}
            d_Z(x_1,x_2)+d_Z(x_2,y_1)&=d_X(x_1,x_2)+r+\inf_{(x',y') \in R} d_X(x_2,x')+d_Y(y',y_1) \\
            &\geq r+\inf_{(x',y') \in R} d_X(x_1,x')+d_Y(y',y_1) =d_Z(x_1,y_1).
        \end{split}
    \end{equation}
    Similarly, $d_Z(x_1,y_1)+d_Z(y_1,y_2) \geq d_Z(x_1,y_2)$. We also have
    \begin{equation}
        \begin{split}
            d_Z(x_1,y_1)+d_Z(y_1,x_2)&=2r+\inf_{(x',y'),(x'',y'') \in R} d_X(x_1,x')+d_X(x_2,x'')+d_Y(y_1,y')+d_Y(y_1,y'') \\
            &\geq 2r+\inf_{(x',y'),(x'',y'') \in R} d_X(x_1,x')+d_X(x_2,x'')+d_Y(y',y'') \\
            &\geq \inf_{(x',y'),(x'',y'') \in R} d_X(x_1,x')+d_X(x_2,x'')+d_X(x',x'') \\
            &\geq d_X(x_1,x_2)=d_Z(x_1,x_2).
        \end{split}
    \end{equation}
    Similarly, $d_Z(y_1,x_1)+d_Z(x_1,y_2) \geq d_Z(y_1,y_2)$. Hence, $d_Z$ is a metric on $Z$. Moreover, $Z$ is complete because any sequence in $Z$ has a subsequence contained in either $X$ or $Y$, which are complete. Since the union of countable dense sets in $X$ and $Y$ is dense in $Z$, it follows that $Z$ is Polish. Lastly, we show that $\pi_Z$ is $1$-Lipschitz. For $x \in X$, $y \in Y$, we have
    \begin{equation}
        \begin{split}
            d_B(\pi_X(x),\pi_Y(y)) &\leq \inf_{(x',y') \in R} d_B(\pi_X(x),\pi_X(x'))+d_B(\pi_X(x'),\pi_Y(y'))+d_B(\pi_Y(y),\pi_Y(y'))\\
            &\leq r+\inf_{(x',y') \in R} d_X(x,x')+d_Y(y,y')=d_Z(x,y).
        \end{split}
    \end{equation}
This completes the proof.
\end{proof}

\begin{lemma}\label{lem:connect_hausdorff}
    Let $\mm{X}$ and $\mm{Y}$ be $B$-fields, $R$ be a correspondence between $X$ and $Y$,  $r>0$ and $r \geq \dis_{\pi_X,\pi_Y}(R)/2$. If $\mm{Z}:=\mm{X} \coprod_{R,r} \mm{Y}$, then $\dha^Z(X,Y)=r$.
\end{lemma}
\begin{proof}
    For $x \in X$, $y \in Y$, we have  $d_Z(x,y) \geq r$, hence $\dha^Z(X,Y) \geq r$. Since $R$ is a correspondence, for any $x \in X$, there exists $y \in Y$ such that $(x,y) \in R$, so $d_Z(x,y)=r$. Similarly, for any $y \in Y$, there exists $x \in X$ such that $(x,y) \in R$, so $d_Z(x,y)=r$. Hence, we get $\dha^Z(X,Y) \leq r$, implying $\dha^Z(X,Y)=r$.   
\end{proof}

\begin{theorem}\label{thm:gh_distortion}
    If $\mm{X}=(X,d_X,\pi_X)$ and $\mm{Y}=(Y,d_Y,\pi_Y)$ are $B$-fields, then
    $$\dgh(\mm{X},\mm{Y})=\inf_R \dis_{\piX,\piY}(R)/2, $$
    where the infimum is taken over all correspondences between $X$ and $Y$.
\end{theorem}
\begin{proof}
    Let $r>\dgh(\mm{X},\mm{Y})$. There is a $B$-fields $\mm{Z}$ and isometric embeddings $\Phi \colon \mm{X} \dto \mm{Z}$ and $\Psi\colon \mm{Y} \dto \mm{Z}$ over $B$ such that $\dha(\phi(X),\psi(Y)) < r$. For $x \in X$ and $y \in Y$, abusing notation, we write $d_Z(x,y)$ to denote $d_Z(\phi(x), \psi(y))$. We also write $\pi_Z(x)$ and $\pi_Z(y)$ for $\pi_Z(\phi(x))$ and $\pi_Z(\psi(y))$, respectively. Let $R$ be the correspondence between $X$ and $Y$ given by
    \begin{equation}
    R:=\{(x,y) \in X \times Y: d_Z(x,y)<r \}.
    \end{equation}
    We have
    \begin{equation}
        |d_X(x,x')-d_Y(y,y')| \leq d_Z(x,y)+d_Z(x',y') \leq 2r
        \end{equation}
        and
    \begin{equation}
        d_B(\pi_X(x),\pi_Y(y)) = d_B(\pi_Z(x),\pi_Z(y))\leq d_Z(x,y) \leq r.
    \end{equation}
    Hence, $\dis_{\pi_X,\pi_Y}(R) < 2r.$ Since $r>\dgh(\mm{X},\mm{Y})$ is arbitrary, we get
    \begin{equation}
    \dgh(\mm{X},\mm{Y}) \geq \inf_R \dis_{\piX,\piY}(R)/2.
    \end{equation}
    For the converse inequality, let $R$ be a correspondence between $X$ and $Y$ and $r \geq \dis_{\pi_X,\pi_Y}(R)/2$. By Lemma \ref{lem:connect_hausdorff}, there exists a $B$-field
    $\mm{Z}$ containing isometric copies of $\mm{X},\mm{Y}$ such that $\dha^Z(X,Y) = r$. Hence, $\dgh(\mm{X},\mm{Y}) \leq r$. Since the correspondence $R$ and $r>\dis_{\pi_X,\pi_Y}(R)/2$ are arbitrary, we have
    \begin{equation}
    \dgh(\mm{X},\mm{Y}) \leq \inf_R \dis_{\piX,\piY}(R),
    \end{equation}
    as claimed.
\end{proof}

Our next goal is to establish the existence of an optimal correspondence that realizes the Gromov-Hausdorff distance between compact fields.
\begin{proposition}\label{prop:gh_realization}
    If $\mm{X}, \mm{Y}$ are compact $B$-fields, then there exists a correspondence $R$ between $X$ and $Y$ such that $$\dgh(\mm{X},\mm{Y})=\dis_{\pi_X,\pi_Y}(R)/2\,.$$
\end{proposition}
\begin{proof}
    Endow $X \times Y$ with the product sup metric and let $\calC$ be set of all closed subspaces of $X \times Y$ equipped with the Hausdorff distance. We claim that $\dis_{\pi_X,\pi_Y}(\cdot): \calC \to [0,\infty)$ is $4$-Lipschitz. Indeed, let $S,T \in \calC $, $\epsilon>\dha(S,T)$ and $(x,y), (x',y') \in S$. Then, there exist $(w,z),(w',z') \in T$ such that $d_X(x,w),d_X(x',w')$, $d_Y(y,z)$, and $d_Y(y',z') < \epsilon$. Thus, we have
    \begin{equation}
        \begin{split}
            |d_X(x,x')-d_Y(y,y')| &\leq |d_X(x,x')-d_X(w,w')|+|d_X(w,w')-d_Y(z,z')|+|d_Y(z,z')-d_Y(y,y')| \\
                                  &\leq d_X(x,w)+d_X(x',w')+\dis_{\pi_X,\pi_Y}(T)+d_Y(z,y)+d_Y(z',y') \\
                                  &\leq \dis_{\pi_X,\pi_Y}(T)+4\epsilon\,.
        \end{split}
    \end{equation}
    We also have
    \begin{equation}
        \begin{split}
            d_B(\pi_X(x),\pi_Y(y)) &\leq d_B(\pi_X(x),\pi_X(w)) + d_B(\pi_X(w),\pi_Y(z)) + d_B(\pi_Y(z),\pi_Y(y)) \\
            &\leq \dis_{\pi_X,\pi_Y}(T)/2+2\epsilon.
        \end{split}
    \end{equation}
    Since $(x,y), (x',y') \in S$ and $\epsilon>\dha(S,T)$ are arbitrary, we have $\dis_{\pi_X,\pi_Y}(S) \leq \dis_{\pi_X,\pi_Y}(T)+4\dha(S,T)$. Similarly, we can show that $\dis_{\pi_X,\pi_Y}(T) \leq \dis_{\pi_X,\pi_Y}(R)+4\dha(S,T)$. Therefore,
    \begin{equation}
        |\dis_{\pi_X,\pi_Y}(S)-\dis_{\pi_X,\pi_Y}(T)| \leq 4\dha(S,T)\,;
    \end{equation}
    that is, the distortion function is 4-Lipschitz. By Theorem \ref{thm:gh_distortion}, there exists a correspondence $R_n$ between $X$ and $Y$ such that
    \begin{equation}
    \dis_{\pi_X,\pi_Y}(R_n) < 2\dgh(\mm{X},\mm{Y})+ 1/n.
    \end{equation}
    Since the distortion of the closure of a relation is same as the distortion of the original relation, without loss of generality, we can assume that $R_n$ is closed. By \cite[Theorem~7.3.8]{burago2001course}, $\calC$ is compact so that, by passing to a subsequence, we can assume that $R_n$ converges to $R$ in the Hausdorff metric. This implies that $\dis_{\pi_f,\pi_g}(R)=2\dgh(\mm{X},\mm{Y})$ by (Lipschitz) continuity. 
    
    To conclude the proof, we show that the projections of $R$ to $X$ and $Y$ are surjective. For each $x \in X$, there exists $y_n \in Y$ such that $(x,y_n) \in R_n$. By compactness, we can assume that $y_n$ converges to $y \in Y$. Given any $\epsilon>0$, $(x,y_n)$ is in the closed $\epsilon$-neighborhood $R^\epsilon$ of $R$ for $n$-large enough, which implies that $(x,y) \in R^\epsilon$. Since $R$ is closed and $\epsilon>0$ is arbitrary, $(x,y) \in R$. Similarly, for any $y \in Y$, there exists $x \in X$ such that $(x,y) \in R$. Hence, $R$ is a correspondence.
\end{proof}

Since the cardinality of a compact metric space is at most that of a continuum, the isometry classes of compact $B$-fields form a set, which we denote by $\fmm{F}_B$. 

\begin{theorem}\label{thm:gh_metric}
   The Gromov-Hausdorff distance $\dgh$ metrizes the moduli space $\fmm{F}_B$ of compact $B$-fields and $(\fmm{F}_B,\dgh)$ is a Polish metric space.
\end{theorem}

We begin the proof with a lemma.

\begin{lemma}\label{lem:finite_approximation}
    Let $\mm{X}$ be a compact $B$-field and $B_0$ a dense subset of $B$. For each $\epsilon>0$, there exists a finite $B$-field $\mm{Y} = (Y, d_Y, \pi_Y)$ such that $\pi_Y$ takes values in $B_0$, $d_Y$ only takes rational values, and $\dgh(\mm{X},\mm{Y}) < \epsilon$.
\end{lemma}
\begin{proof}
    Let $Y$ be a finite subset of $X$ such that for all $x \in X$ there exists $y \in Y$ satisfying $d_X(x,y)<\epsilon/3$. For each $y \in Y$, pick $b_y \in B_0$ such that $d_B(\pi_X(y),b_y)<\epsilon/3$. Letting $n>0$ be an integer such that $1/n<\epsilon/3$, define $d_Y: Y \times Y \to \mathbb{Q}$ by
    \begin{equation}
    d_Y(y,y'):=\lceil n \max(d_X(y,y'),d_B(b_y,b_y')) \rceil/n.
    \end{equation}
    Symmetry and definiteness of $d_Y$ are clear, so to show that it is a metric it suffices to verify the triangle inequality:
    \begin{equation}
        \begin{split}
            d_Y(y,y')+d_Y(y',y'') &= \big( \lceil n \max\{d_X(y,y'),d_B(b_y,b_{y'})\} \rceil +  \lceil n \max\{d_X(y',y''),d_B(b_{y'},b_{y''}\}) \rceil \big)/n \\
            &\geq  \lceil n (\max\{d_X(y,y'),d_B(b_y,b_{y'})\} + \max\{d_X(y',y''),d_B(b_{y'},b_{y''})\} )\rceil/n \\
            &\geq \lceil n \max\{d_X(y,y')+d_X(y',y''),d_B(b_y,b_{y'})+d_B(b_{y'},b_{y''})\}\rceil /n \\
            &\geq  \lceil n \max\{d_X(y,y''),d_B(b_y,b_{y''})\}\rceil /n=d_Y(y,y'').
        \end{split}
    \end{equation}
    Define $\pi_Y: Y \to B$ by $y \mapsto b_y$. Note that, by definition, $\pi_Y$ is $1$-Lipschitz and takes values in $B_0$. To conclude, we show that the $B$-field $\mm{Y}:=(Y,d_Y,\pi_y)$ is $\epsilon$-close to $\mm{X}$. Let $R$ be the correspondence between $X$ and $Y$ given by
    \begin{equation}
    R:=\{(x,y): x \in X, y \in Y, d_X(x,y)<\epsilon/3\}.
    \end{equation}
    It is enough to show that $\dis_{\pi_X,\pi_Y}(R) \leq 2\epsilon$. Let $(x,y), (x',y') \in R$. Note that $d_Y(y,y') \geq d_X(y,y')$. Using
    $d_B(b_y,b_{y'}) \leq 2\epsilon/3 +d_B(\pi_X(y),\pi_X(y')) \leq   2\epsilon/3 +d_X(y,y')$, we can get
    \begin{equation}
        \begin{split}
            |d_Y(y,y')-d_X(x,x')| &\leq d_Y(y,y')-d_X(y,y')+|d_X(y,y')-d_X(x,x')| \\
            &\leq 1/n + \max(d_X(y,y'),d_B(b_y,b_{y'}))-d_X(y,y')+d_X(x,y)+d_X(x',y') \\
            &\leq \max(d_X(y,y'),d_X(y,y')+2\epsilon/3)-d_X(y,y')+\epsilon < 2\epsilon.
        \end{split}
    \end{equation}
    We also have that
    \begin{equation}
            d_B(\pi_X(x),\pi_Y(y)) \leq d_B(\pi_X(x),\pi_X(y))+d_B(\pi_X(y),\pi_Y(y)) \leq d_X(x,y)+\epsilon/3 <\epsilon \,.
    \end{equation}
    Therefore, $\dis_{\pi_X,\pi_Y}(R) \leq 2\epsilon$. This completes the proof.
\end{proof}

\begin{proof}[Proof of Theorem \ref{thm:gh_metric}]
    Symmetry of $\dgh$ is straightforward and the triangle inequality has been established in Proposition \ref{prop:gh_triangle}. To show definiteness, suppose that $\dgh(\mm{X},\mm{Y})=0$. We need to show that $\mm{X}$ is isometric to $\mm{Y}$. By Proposition \ref{prop:gh_realization}, there exists a correspondence $R$ between $X$ and $Y$ such that $\dis_{\pi_X,\pi_Y}(R)=0$. This implies that $R$ is the graph of an isometry between $\mm{X}$ and $\mm{Y}$.
    
    We now argue that $(\fmm{F}_B,\dgh)$ is Polish. Let $B_0$ be a countable dense set in $B$ and denote by $\calF$ the collection of isometry classes of finite $B$-fields that map into $B_0$ and only have rational distances. By Lemma \ref{lem:finite_approximation}, $\calF$ is dense in $\fmm{F}_B$. Note that a subset of the countable set $\cup_{n>0} (\mathbb{Q}^{n^2} \times B_0^n)$ surjects onto $\calF$, which implies that $\calF$ is countable.

    It remains to verify completeness. Let $\mm{X}_n$ be a Cauchy sequence of compact $B$-fields with respect to Gromov-Hausdorff distance. To prove that the sequence is convergent, it suffices to construct a convergent subsequence. By passing to a subsequence if necessary, we can assume that $\dgh(\mm{X}_n,\mm{X}_{n+1})<1/2^n$ for all $n$. Then, there exists a correspondence $R_n$ between $X_n$ and $X_{n+1}$ such that $\dis_{\pi_{X_n},\pi_{X_{n+1}}}(R_n) < 2/2^n$ for all $n$. Letting $r_n=1/2^n$, apply the construction described in Definition \ref{def:connecting_fields} consecutively to get a $B$-field
    \begin{equation}
    \mm{Z}_n:=\Big( \big( (\mm{X}_1 \coprod_{R_1,r_1} \mm{X}_2) \coprod_{R_2,r_2} \mm{X}_3\big) \ldots\Big) \coprod_{R_n,r_n} \mm{X}_{n+1} .
    \end{equation}
    for each $n>0$. Clearly, $\mm{Z}_n$ is a subfield of $\mm{Z}_{n+1}$. Let $\mm{Z}$ be the completion of the co-limit $\cup_{n>0} \mm{Z}_n$. Note that the union of countable dense sets in $Z_n$ is a countable dense set in $Z$ and therefore $Z$ is Polish. Hence, $\mm{Z}$ is a $B$-field. By Lemma \ref{lem:connect_hausdorff}, $\dha^Z(X_n,X_{n+1})=r_n=1/2^n$ so $(X_n)$ forms a Cauchy sequence with respect to Hausdorff distance in $Z$. By \cite[Proposition~7.3.7]{burago2001course}, there exists a closed subspace $Y \subseteq Z$ such that $X_n$ Hausdorff converges to $Y$ in $Z$. Now we show that $Y$ is compact. Since $Y$ is complete, it is enough to show that it is totally bounded. Given $\epsilon>0$, pick $n>0$ such that $\dha^Z(X_n,Y) < \epsilon/3$ and let $A \subseteq X_n$ be a finite $\epsilon/3$-net of $X_n$. For each $a \in A$, let $b_a \in Y$ be such that $d_Z(a,b_a)<\epsilon/3$ and set $B:=\{b_a:a \in A \}$. For any $y \in Y$, there exist $x \in X_n$ such that $d_Z(x,y)<\epsilon/3$ and $a \in A$ such that $d_Z(x,a) \leq \epsilon/3$. Therefore, we have
    \begin{equation}
    d_Z(y,b_a) \leq d_Z(y,x)+d_Z(x,a)+d_Z(a,b_a) < \epsilon \,.
    \end{equation}
    This means that $B$ is a finite $\epsilon$-net in $Y$ so that $Y$ is totally bounded. Hence, $Y$ is compact. Letting $\mm{Y}=(Y,d_Z|_Y,\pi_Z|_Y)$, we have that $\dgh(\mm{X}_n,\mm{Y}) \leq \dha^Z(X_n,Y)$. Thus, $\mm{X}_n$ converges to $\mm{Y}$ in the Gromov-Hausdorff distance.
\end{proof}

It is known that the isometry classes of compact metric spaces form a geodesic space under the Gromov-Hausdorff distance \cite{ivanov2016gromov, Chowdhury2018}. (A geodesic space is a metric space such that between any two points, there is a shortest path whose length is equal to the distance between its endpoints.) We close this section with an extension of the result to $B$-fields. We note that our method of proof is distinct from that of \cite{ivanov2016gromov}. Nonetheless, with the additional assumption that $B$ has a geodesic bicombing, the proof in \cite{ivanov2016gromov} can be adapted to this setting.

\begin{proposition} \label{prop:geo}
   $(\fmm{F}_B,\dgh)$ is a geodesic space if and only if $B$ is a geodesic space.
\end{proposition}
\begin{proof}
    Assume $(\fmm{F}_B,\dgh)$ is a geodesic space. To show that this implies that $B$ is geodesic, by \cite[Theorem~2.4.16]{burago2001course}, it suffices to prove that any pair of distinct points $b, b' \in B$ has a midpoint. Let $X=Y:=\{\ast\}$ be the same 1-point space (with the $0$ metric) and set $\mm{X}=(\{\ast\}, 0, \ast \mapsto b)$, $\mm{Y}=(\{\ast\}, 0, \ast \mapsto b')$. There is a unique correspondence $R$ between $X$ and $Y$ whose distortion is $\dis_{\pi_X,\pi_Y} (R)=2d_B(b,b')$. By Proposition \ref{thm:gh_distortion}, $\dgh(\mm{X},\mm{Y})=d_B(b,b')$. By assumption, there exists a compact field $\mm{W}$ such that $\dgh(\mm{X},\mm{W})=\dgh(\mm{Y},\mm{W})=d_B(b,b')/2$. Let $S$ and $T$ be the unique correspondences between $X$ and $W$, and $Y$ and $W$, respectively. By Proposition \ref{thm:gh_distortion}, $\dis_{\pi_X,\pi_W} (S)=\dis_{\pi_Y,\pi_W} (T)=d_B(b,b')$. Then, for any $w \in W$, we have $d_B(b,\pi_W(w)) \leq d_B(b,b')/2$ and $d_B(b',\pi_W(w)) \leq d_B(b,b')/2$. This implies that $\pi_W(w)$ is a midpoint between $b$ and $b'$.

    To prove the converse statement, assume $B$ is a geodesic space. Given any pair of compact $B$-fields $\mm{X}$ and $\mm{Y}$, we construct a Gromov-Hausdorff geodesic between them. We can assume that $r=\dgh(\mm{X},\mm{Y}) > 0$. By Proposition \ref{prop:gh_realization}, there is a correspondence $R$ between $X$ and $Y$ such that $\dis_{\pi_X,\pi_Y}(R)=2r$. Let $\mm{W}:=\mm{X} \coprod_{R,r} \mm{Y}$, Then, both $\mm{X}$ and $\mm{Y}$ are contained in $\mm{W}$ and $\dha^W(X,Y)=r$ by Lemma \ref{lem:connect_hausdorff}. Now we use the fact that every compact metric space can be isometrically embedded into a compact geodesic space, for example, its injective hull, (cf.\,\cite{lang2013injective}). Let $\overline{W}$ be a compact geodesic space containing $W$ and $Z=\overline{W} \times B$ endowed with the product sup-metric. $Z$ is a geodesic space because both $\overline{W}$ and $B$ have this property. Letting $\pi_Z: Z \to B$ denote projection onto the second coordinate, $\mm{Z}:=(Z,B,\pi_Z)$ defines a $B$-field. Moreover, $\mm{W}$ isometrically embeds into $\mm{Z}$ via $w \mapsto (w,\pi_W(w))$. Thus, we can assume that $\mm{X}$ and $\mm{Y}$ are sub-fields of $\mm{Z}$ satisfying $\dha^Z(X,Y)=r$. Define a correspondence between $X$ and $Y$ by
    \begin{equation}
     R:=\{(x,y): \text{$x$ is a $d_Z$-closest point to  $y$ in  $X$ or  $y$ is a $d_Z$-closest point to $x$ in  $Y$}\}.
    \end{equation}
    Any $(x,y) \in R$ satisfies $d_Z(x,y) \leq r$. For $(x,y) \in R$, let $\gamma_{x,y}\colon [0,1] \to Z$ be a constant speed geodesic from $x$ to $y$ in $Z$. For $t \in [0,1]$, let $A_t=\{\gamma_{x,y}(t): (x,y) \in R \}$. Since $R$ is a correspondence, we have $A_0=X$ and $A_1=Y$. Since
    \begin{equation}
    d_Z(\gamma_{x,y}(s),\gamma_{x,y}(t))=|s-t|\, d_Z(x,y) \leq |s-t|\, r \,,
    \end{equation}
    for all $(x,y) \in R$ and $s,t \in [0,1]$, we have
    \begin{equation}
    \dha^Z(A_s,A_t) \leq |s-t|\, \dgh(\mm{X},\mm{Y}) \,.
    \end{equation}
    Let $\mm{X}_t:=(\bar{A}_t,d_Z|_{\bar{A}_t},\pi_Z|_{\bar{A}_t})$, where $\bar{A}_t$ denotes the closure of $A_t$ in $Z$. Then, $\mm{X}_t$ is a compact $B$-field, $\mm{X}_0=\mm{X}$, $\mm{X}_1=\mm{Y}$ and 
    \begin{equation}
    \dgh(\mm{X}_s,\mm{X}_t) \leq \dha^Z(\bar{A}_s,\bar{A}_t) =\dha^Z(A_s,A_t) \leq |s-t|\,\dgh(\mm{X},\mm{Y}),
    \end{equation}
    for all $s, t \in [0,1]$, which implies that $ \dgh(\mm{X}_s,\mm{X}_t)= |s-t|\,\dgh(\mm{X},\mm{Y})$ by the triangle inequality. This shows that $t \mapsto \mm{X}_t$ is a geodesic between $\mm{X}$ and $\mm{Y}$ in $(\fmm{F}_B,\dgh)$.
\end{proof}

\section{The Urysohn Field} \label{sec:urysohn}

The primary goal of this section is to obtain a description of the Gromov-Hausdorff distance for compact $B$-fields in terms of the Hausdorff distances between subfields of a Urysohn universal $B$-field modulo the action of its automorphism group. An {\em automorphism} of a field $\pi \colon X \to B$ is a bijective isometry $\psi \colon X \to X$ that satisfies $\pi \circ \psi = \pi$. An analogous result holds for the Gromov-Prokhorov distance.

\begin{definition}[Urysohn Field]
A $B$-field $\pi\colon U \to B$ is called a {\em Urysohn field over $B$} if for each finite  subspace $A$ of $U$ and $1$-Lipschitz map $\phi\colon A^\ast \to B$ defined on a one-point metric extension $A^\ast= A \sqcup \{a^\ast\}$, satisfying $\phi|_A = \pi|_A$, there exists  an isometric embedding  $\imath \colon A^* \to U$ such that the restriction $\imath|_A$ is the inclusion map and $\pi \circ \imath = \phi$.
\end{definition}

The next theorem is a special case of a more general result proven in \cite{doucha2013universal} in a model theory framework. For a proof based on metric geometry constructs that extends to the functional setting a well-known construction of Urysohn space due to Kat{\v e}tov \cite{katetov1986universal}, the reader may consult \cite{anbouhi2022universal}.

\begin{theorem}[Existence and Uniqueness of Urysohn Fields]\label{thm:urysohnmap}
If $B$ is a Polish space, then the following statements hold:
\begin{enumerate}[\rm (i)]
    \item there exists a Urysohn field $\pi\colon U \to B$, unique up to isometry;
    \item Urysohn field is universal, that is, every $B$-field isometrically embeds into the Urysohn field;
    \item every isometry between finite subfields of the Urysohn field extends to an automorphism of the Urysohn field.
\end{enumerate}
\end{theorem}

Given an equivalence relation $\sim$ on a metric space $(X,d_X)$, the quotient metric is the maximal (pseudo) metric on $X/\!\!\sim$ that makes the quotient map $\pi \colon X \to X/\!\!\sim$ 1-Lipschitz. Let $F(\mm{U}_B)$ be the space of compact subfields of $\mm{U}_B = (U,B,\pi_U)$, equipped with the Hausdorff distance, and $Aut(B)$ the automorphism group of $\mm{U}_B$, which acts on $\mm{U}_B$ by isometries. On $F(\mm{U}_B)/Aut(B)$, by \cite[Lemma~3.3.6]{burago2001course}, the quotient metric may be expressed as
\begin{equation} \label{eq:autb}
d_F^B (\mm{X},\mm{Y}) =  \inf_{\Phi, \Psi \in Aut(B)} d_H^U (\phi (X),\psi (Y))) = \inf_{\Psi \in Aut(B)} d_H^U (X,\psi(Y)).
\end{equation}

The next result interprets $(\fmm{F}_B, \dgh)$ in terms of subfields of a Urysohn universal field. There is a corresponding result for compact metric spaces, as remarked in \cite[3.11~2/3]{gromov2007metric}.

\begin{theorem} \label{thm: F_B = F_I}
The moduli space $(\fmm{F}_B, \dgh)$ of isometry classes of compact $B$-fields equipped with the Gromov-Hausdorff distance is isometric to the quotient space $(F(\mm{U}_B)/Aut(B), d_F^B)$.
\end{theorem}

We provide a proof at the end of this section after establishing some results needed for the argument.

\begin{lemma} \label{GH in U}
Let $\mm{X}=(X,B,\pi_X)$ and $\mm{Y}=(Y,B,\pi_Y)$ be compact $B$-fields and $\mm{U}_B=(U,B,\pi_U)$ be the Urysohn $B$-field. Then,
\begin{equation*}
    \dgh(\mm{X}, \mm{Y}) = \inf_{\Phi,\Psi} d^U_H(\phi(X),\psi (Y)),
\end{equation*}
where the infimum is taken over all isometric embeddings $\Phi\colon \mm {X} \dto \mm{U}_B$ and $\Psi \colon \mm{Y} \dto \mm{U}_B$ over $B$. The corresponding result for the Gromov-Prokhorov distance also holds.
\end{lemma}
\begin{proof}
The inequality $ \dgh(\mm{X}, \mm{Y}) \leq \inf_{\Phi,\Psi} d^U_H(\phi(X),\psi (Y))$ follows from the definition of $\dgh$. To prove the reverse inequality, let $\epsilon >0$. We show that there are isometric embeddings  $\Phi \colon \mm{X} \dto \mm{U}_B$ and $\Psi \colon \mm{Y} \dto \mm{U}_B$ such that 
\begin{equation}
d^U_H(\phi(X), \psi (Y)) \leq \dgh(\mm{X}, \mm{Y}) +\epsilon \,.
\end{equation}
By definition of the Gromov-Hausdorff distance, there is a $B$-field $\mm{Z}$ and isometric embeddings $\Phi' \colon \mm{X} \dto \mm{Z}$ and $\Psi' \colon \mm{Y} \dto \mm{Z}$ over $B$ such that
\begin{equation}
d^Z_H(\phi'(X), \psi'(Y)) \leq \dgh(\mm{X},\mm{Y}) +\epsilon \,.
\end{equation}
By the universality of $\mathscr{U}_B$, there is an isometric embedding $\Lambda \colon \mm{Z} \dto \mm{U}_B$ over $B$. Letting $\Phi =\Lambda\circ\Phi'$ and $\Psi=\Lambda\circ\Psi'$, we have
\begin{equation} \label{eq:ghineq}
d^U_H(\phi(X), \psi(Y)) = d^Z_H(\phi'(X), \psi'(Y)) \leq \dgh (\mm{X},\mm{Y}) +\epsilon \,.
\end{equation}
Clearly, the same inequality holds if the left-hand side of \eqref{eq:ghineq} is replaced with the infimum over $\Phi$ and $\Psi$. Taking the limit as $\epsilon \to 0$, we obtain the desired inequality.

\end{proof}

\begin{lemma}\label{compact injectivity}
Let $\mm{U}_B$ be a Urysohn field over $B$, $A \subseteq U$ a compact subset, and $f \colon A^\ast \to B$ a field defined on a one-point metric extension $A^\ast= A \sqcup \{a^\ast\}$ satisfying $f_A = \pi|_A$. Then, there exists an isometric embedding $\imath \colon A^\ast \to U$ over $B$ such that the restriction $\imath|_A$ is the  inclusion map and $\pi \circ \imath = f$. 
\end{lemma}
\begin{proof}
It suffices to construct a sequence $(u_n)$ in $U$ satisfying
\begin{enumerate}[(i)]
\item $\pi_U(u_n) = f(a^\ast)$, for $n \geq 1$;
\item $|d_\ast(a^\ast, a)-d_U(u_n,a)|\leq 2^{-n}$, $\forall a\in A$, where $d_\ast$ is the metric on $A^\ast$;
\item $d_U(u_n,u_{n+1})\leq 2^{-n}$, for $n \geq 1$.
\end{enumerate}
Indeed, letting $u=\lim_{n\to \infty} u_n \in U$, define $\imath \colon A^\ast \to B$ as the identity on $A$ and $\imath (a^\ast) =u$. The map $\imath$ gives the desired one-point extension. Now we proceed to the construction of the sequence $(u_n)$.

Let $A_n$ be an ascending sequence of finite subsets of $A$, where $A_n$ is a $2^{-(n+1)}$-net in $A$, for $n \geq 1$. (This means that the balls of radius $2^{-(n+1)}$ centered at the points in $A_n$ cover $A$.) Let $D_1=A_1$ and denote by $D_1^\ast = D_1\sqcup\{a^\ast\}$, the one-point metric extension of $D_1$ induced by $(A^\ast, d_\ast)$. Since $\mm{U}_B$ is Urysohn, applying the one-point extension property to the field $f|_{D_1^\ast}$, we obtain an isometric embedding $\imath_1 \colon D_1^\ast \to U$ such that $\pi_U \circ \imath_1 = f|_{D_1^\ast}$. Defining $u_1 = \imath_1 (a^\ast) \in U$, we have that $\pi_U (u_1) = f(a^\ast)$ and $d_\ast (a^\ast, a) = d_U (u_1, a)$, for any $a \in A$, so $u_1$ satisfies (i) and (ii) above. Condition (iii) is empty at this stage of the construction. Inductively, suppose that we have constructed  $u_j$, $i \leq j \leq n$, with the desired properties and let
\begin{equation}
D_{n+1} = A_{n+1} \cup \{u_n\} \quad \text{and} \quad D_{n+1}^\ast=  D_{n+1} \sqcup\{a^\ast\}.
\end{equation}
Using the notation $A_{n+1}^\ast = A_{n+1} \cup \{a^\ast\}$, define a metric $d'_\ast \colon D_{n+1}^\ast \times D_{n+1}^\ast \to \real$, as follows: $d'_\ast$ coincides with $d_\ast$ on $A_{n+1}^\ast \times A_{n+1}^\ast $, $d'_\ast (a,u_n) = d_U (a, u_n)$, for every $a \in A_{n+1}$, and
\begin{equation}
d'_\ast(a^\ast, u_n):=\sup_{b\in A_{n+1}}|d_\ast(a^\ast,b)-d_U(u_n,b)|.
\label{eq:metric}
\end{equation}
Define a field $f' \colon D_{n+1}^\ast \to B$ by $f'|_{D_{n+1}} = \pi_U|_{D^{n+1}}$ and $f'(a^\ast) = f(a^\ast)$. Applying the one-point extension property to $f'$ we obtain an isometric embedding $\imath_{n+1} \colon D_{n+1}^\ast \to U$ satisfying $f' = \pi_U\circ \imath_{n+1}$.

Let $u_{n+1} = \imath_{n+1} (a^\ast) \in U$. By construction, $\pi_U(u_{n+1}) = f'(a_\ast) = f (a_\ast)$, so requirement (i) is satisfied. Moreover, $d_U(u_{n+1},b) = d'_\ast(a^\ast, b)= d_\ast(a^\ast, b)$, for any $b \in A_{n+1}$. Since $A_{n+1}$ is a $2^{-(n+2)}$-net in $A$, for each $a \in A$, we can pick $b \in A_{n+1}$ such that $d(a,b) \leq 2^{-(n+2)}$. Then, we have
\begin{equation}
\begin{split}
|d_\ast (a^\ast, a) - d_U(u_{n+1},a)| &\leq |d_\ast (a^\ast, a) - d_U(u_{n+1},b)| + | d_U(u_{n+1},b) - d_U(u_{n+1},a)| \\
&= |d_\ast (a^\ast, a) - d_\ast (a^\ast,b)| + | d_U (u_{n+1},b) - d_U(u_{n+1},a)| \\
&\leq d_\ast(a,b) + d_U (a,b) = 2 d_U (a,b) \leq 2^{-(n+1)} \,.
\end{split}
\end{equation}
This verifies property (ii). By the inductive hypothesis, we also have $|d_\ast (a^\ast, a) - d_U(u_n,a)| \leq 2^{-n}$, for any $a \in A$. Thus, by \eqref{eq:metric},
\begin{equation}
d(u_{n+1}, u_n) = d'_\ast (a^\ast, u_n) = \sup_{b\in A_{n+1}}|d_\ast(a^\ast,b)-d_U(u_n,b)| \leq 2^{-n}.
\end{equation}
This concludes the proof.
\end{proof}

A {\em partial isometric matching} of $\pi \colon X \to B$ is a bijective isometry $\phi \colon A \to A'$ between  subspaces of $X$ that satisfies $\pi|_{A'} \circ \phi = \pi|_A$.

\begin{proposition}\label{compact homogeneity}
If $\mm{U}_B$ is a Urysohn field and $A, B \subseteq U$ are compact subsets, then any partial isometric matching $\phi \colon A \to B$ of $\mm{U}_B$ extends to an automorphism of $\mm{U}_B$.
\end{proposition}
\begin{proof}
Let $C=\{x_1, x_2, \ldots \}$ and $C'=\{x'_1, x'_2, \ldots\}$ be countable dense sets in $U\setminus A$ and $U \setminus B$, respectively. Using the one-point compact extension property established in Lemma \ref{compact injectivity} and a back-and-forth argument applied to $C$  and $C'$ as in \cite[Section~3.1]{huvsek2008urysohn}, $\phi$ can be extended to an isometry of $\mm{U}_B$.
\end{proof}

\begin{proof}[Proof of Theorem \ref{thm: F_B = F_I}]
For any compact $B$-fields $\mm{X}$ and $\mm{Y}$, Lemma \ref{GH in U} shows that 
\begin{equation}
    \dgh(\mm{X},\mm{Y})= \inf_{\Phi, \Psi} d^{U}_H(\phi(X),\psi(Y)),
\end{equation}
where $\Phi \colon \mm{X} \dto \mm{U}_B$ and $\Psi \colon \mm{Y} \dto \mm{U}_B$ are isometric embeddings. By Proposition \ref{compact homogeneity}, any other isometric embeddings $\Phi' \colon \mm{X} \dto \mm{U}_B$ and $\Psi' \colon \mm{Y} \dto \mm{U}_B$ differ from $\Phi$ and $\Psi$ by the action of automorphisms of $\mm{U}_B$. This proves the claim.
\end{proof}


\section{Gromov-Prokhorov Distance for Metric-Measure Fields} \label{sec:gp}

As in Section \ref{sec:gh}, for completeness, this section discusses field versions of analogous results for metric-measure spaces (cf.\,\cite{loehr2013equivalence}, \cite{janson2020gromov}, \cite[3~1/2]{gromov2007metric}).Recall that a metric measure space ($mm$-space) is a triple $(X, d_X, \mu_X)$, where $(X, d_X)$ is a Polish metric space and $\mu_X$ is a Borel probability measure on $X$. The next definition introduces an analogue for fields.

\begin{definition}
Let $B$ be a Polish metric space. A {\em metric measure field}, or $mm$-field, over $B$ is a quadruple $\mm{X} = (X,d_X,\pi_X,\mu_X)$, where $(X,d_X)$ is a Polish metric space, $\pi_X \colon X \to B$ is a $1$-Lipschitz map, and $\mu_X$ is a Borel probability measure on $X$. Two $mm$-fields are {\em isomorphic} if there is a measure-preserving isometry between the underlying $B$-fields.
\end{definition}
We abuse notation an also denote by $\mm{X}$ the $B$-field underlying $\mm{X} = (X,d_X,\pi_X,\mu_X)$.

\begin{definition} (Gromov-Prokhorov Distance for Fields) \label{def: functional GPB distance}
Let $\mm{X}=(X,d_X,\pi_X,\mu_X)$ and $\mm{Y}=(Y,d_Y,\pi_Y,\mu_Y)$ be $mm$-fields over $B$. The {\em Gromov-Prokhorov distance} is defined by
\begin{equation*}
    \dgp(\mm{X}, \mm{Y}) := \inf_{Z,\Phi,\Psi} \dpr^Z(\phi_*(\mu_X),\psi_*(\mu_Y)),
\end{equation*}
where the infimum is taken over all $B$-fields $\mm{Z}$ and isometric embeddings $\Phi \colon \mm{X} \dto \mm{Z}$ and $\Psi\colon \mm{Y} \dto \mm{Z}$ over $B$. 
\end{definition}

\begin{proposition}\label{prop:gp_triangle}
    Let $\mm{X}$, $\mm{Y}$, and $\mm{W}$ be $mm$-fields over $B$. Then, 
    $$\dgp(\mm{X},\mm{W}) \leq \dgp(\mm{X},\mm{Y})+\dgp(\mm{Y},\mm{W}). $$
\end{proposition}
\begin{proof}
    The proof is identical to that of Proposition \ref{prop:gh_triangle} replacing the Hausdorff distance by the Prokhorov distance.
\end{proof}

Given $(X, d_X, \muX)$ and $(Y, d_Y, \muY)$, let $\cXY$ denote the set of all couplings between $\muX$ and $\muY$; that is, the collection of all probability measures $\mu$ on $X \times Y$ that marginalize to $\mu_X$ and $\mu_Y$, respectively. Our next goal is to characterize the Gromov-Prokhorov distance between $B$-fields in terms of couplings.

\begin{definition}[$\epsilon$-couplings]
    Let $\mm{X}=(X,d_X,\pi_X,\mu_X)$ and $\mm{Y}=(Y,d_Y,\pi_Y,\mu_Y)$ be $mm$-fields over $B$ and $\epsilon \geq 0$. A coupling $\mu \in \cXY$ is called an $\epsilon$-\emph{coupling} if there is a Borel subset $R$ of $X \times Y$ such that $\mu(R) \geq 1-\epsilon/2$ and $\dis_{\pi_X,\pi_Y}(R) \leq \epsilon$. (Note: $R$ is not assumed to be a correspondence.)
\end{definition}

The following is an analogue of Lemma \ref{lem:connect_hausdorff} for the Prokhorov distance.

\begin{lemma}\label{lem:connect_prokhorov}
    Let $\mm{X}$ and $\mm{Y}$ be $mm$-fields over $B$ and $r>0$. Suppose that $\mu$ is a $2r$-coupling between $\mm{X}$ and $\mm{Y}$ with respect to a Borel relation $R \subseteq X \times Y$ and let $\mm{Z}:=\mm{X} \coprod_{R,r} \mm{Y}$. Then, $\dpr^Z(\mu_X,\mu_Y) \leq r$.
\end{lemma}
\begin{proof}
    Given a Borel subset $A \subseteq X$, denote by $A^r$ the closed $r$-neighborhood of $A$ in $Z$. Viewing $\mu$ as a measure on $Z \times Z$ supported on $X \times Y$, (under the inclusion $X\times Y \hookrightarrow Z \times Z$), we have
    \begin{equation}
    \mu_Y(A^r)=\mu(Z \times A^r) \geq \mu((A \times Y) \cap R) \geq \mu(A \times Y)-r = \mu_X(A)-r \,.
    \end{equation}
    Similarly, $\mu_X (A^r) \geq \mu_Y (A)-r$. Hence, $\dpr^Z(\mu_X,\mu_Y) \leq r$.
\end{proof}

\begin{theorem}\label{thm:gp_coupling}
    Let $\mm{X}$ and $\mm{Y}$ be $mm$-fields over $B$. The Gromov-Prokhorov distance between $\mm{X}$ and $\mm{Y}$ can be expressed as
    \begin{equation*}
        \dgp(\mm{X},\mm{Y}) = \frac{1}{2}\inf \, \{\epsilon>0 \colon \text{there exists an $\epsilon$-coupling $\mu \in C(\mu_X, \mu_Y)$}\}.
    \end{equation*}
\end{theorem}
\begin{proof}
    We first show that if $\epsilon>0$ and $\mu$ is an $\epsilon$-coupling between $\mm{X}$ and $\mm{Y}$, then $\dgp(\mm{X},\mm{Y}) \leq \epsilon/2$. If $r>\epsilon/2$, then $\mu$ is a $2r$-coupling between $\mm{X}$ and $\mm{Y}$. By Lemma \ref{lem:connect_prokhorov}, there exists a $B$-field $\mm{Z}$ containing $\mm{X}$ and $\mm{Y}$ such that $\dpr^Z(\mu_X,\mu_Y) \leq r$, which implies that $\dgp(\mm{X},\mm{Y}) \leq r$. Since $r>\epsilon/2$ is arbitrary, $\dgp(\mm{X},\mm{Y}) \leq \epsilon/2$. Thus, if we denote the infimum in the statement of the theorem by $\alpha$, we have that $\dgp(\mm{X},\mm{Y}) \leq \alpha/2$. We now sharpen this to an equality.
    
    Let $r>\dgp(\mm{X}, \mm{Y})$. Let us show that there is a $2r$-coupling between $\mm{X}$ and $\mm{Y}$. There exists isometric embeddings $\Phi \colon \mm{X} \dto \mm{Z}$ and $\Psi\colon \mm{Y} \dto \mm{Z}$ such that $\dpr^Z(\phi_*(\mu_X),\psi_*(\mu_Y)) < r$. For $x \in X$ and $y \in Y$, abusing notation, write $d_Z(x,y)$ for $d_Z(\phi(x),\psi(y))$, and $\pi_Z(x)$ and $\pi_Z(y)$ for $\pi_Z(\phi(x))$ and $\pi_Z(\psi(y))$, respectively. By \cite[Theorem~11.6.2]{dudley2018real}, there exists a coupling $\nu$ between $\phi_*(\mu_X)$ and $\psi_*(\mu_Y)$ such that
    \begin{equation}
    \nu(\{(w,z) \in Z \times Z\colon d_Z(w,z)> r \}) \leq r.
    \end{equation}
    Since $\nu$ is supported in $\phi(X) \times \psi(Y)$, it induces a coupling $\mu$ between $\mu_X$ and $\mu_Y$ such that
    \begin{equation} \label{eq:dprob}
    \mu (\{(x,y) \in X \times Y\colon d_Z(x,y)> r \}) \leq r.
    \end{equation}
    Let $R:=\{(x,y) \in X \times Y\colon d_Z(x,y)\leq  r \}$. By \eqref{eq:dprob}, $\mu(R) \geq 1-r$. Moreover, if $(x,y),(x',y') \in R$, then
    \begin{equation}
            |d_X(x,x')-d_Y(y,y')| \leq d_Z(x,y)+d_Z(x',y') \leq 2r 
    \end{equation}
    and
    \begin{equation}
        |d_B(\pi_X(x),\pi_Y(y)| =d_B(\pi_Z(x),\pi_Z(y)) \leq d_Z(x,y) \leq r.
    \end{equation}
    Hence, $\dis_{\pi_X,\pi_Y}(R) \leq 2r$. This implies that $\dgp(\mm{X},\mm{Y})=\alpha/2$, concluding the proof.
\end{proof}

The next proposition establishes the existence of optimal couplings for the Gromov-Prokhorov distance between compact fields.

\begin{proposition}\label{prop:gp_realization}
    Let $\mm{X}$ and $\mm{Y}$ be compact $mm$-fields over $B$. There exists an $\epsilon$-coupling between $\mm{X}$ and $\mm{Y}$ such that $\dgp(\mm{X},\mm{Y})=\epsilon/2$.
\end{proposition}
\begin{proof}
    Let $\dgp(\mm{X},\mm{Y})=r$. By Theorem \ref{thm:gp_coupling}, for each integer $n>0$, there exists a $(2r+2/n)$-coupling $\mu_n$ between $\mm{X}$ and $\mm{Y}$. By Proposition \ref{prop:couplings}, passing to a subsequence if necessary, there exists a coupling $\mu$ between $\mm{X}$ and $\mm{Y}$ such that $\mu_n$ weakly converges to $\mu$. It is enough to show that $\mu$ is a $2r$-coupling, as we can take $\epsilon = 2r$.

    There exist Borel sets $R_n \subseteq X \times Y$ such that $\mu_n(R_n) \geq 1-r-1/n$ and $\dis_{\pi_X,\pi_Y}(R_n) \leq 2r+2/n$, for all $n$. Without loss of generality, we can assume that $R_n$ is closed. By passing to a subsequence if necessary, by \cite[Theorem~7.3.8]{burago2001course}, there exists a closed subset $R$ of $X \times Y$ such that $R_n$ Hausdorff converges to a closed subset $R$, where $X\times Y$ is endowed with the product sup-metric. As in the proof of Proposition \ref{prop:gh_realization}, we have $\dis_{\pi_X,\pi_Y}(R) \leq 2r$. It remains to show that $\mu(R) \geq 1-r$. Let $\delta>0$ and $R^\delta$ denote the closed $\delta$-neighborhood of $R$ in $X \times Y$. For $n$-large enough, $R_n \subseteq R^\delta$, so $\mu_n(R^\delta) \geq 1-r-1/n$. Hence, by \cite[Theorem~11.1.1]{dudley2018real}, $\mu(R^\delta) \geq 1-r$. Therefore, $\mu(R)=\mu(\cap_n R^{1/n})=\lim_{n \to \infty}\mu(R^{1/n}) \geq 1-r.$
\end{proof}

Two $mm$-fields over $B$ are called isomorphic if there is a measure preserving isometry between them. The following result shows that, even without a compactness requirement, fully supported fields with zero Gromov-Prokhorov distance are isomorphic.

\begin{proposition}\label{prop:gp_isom}
    Let $\mm{X}$ and $\mm{Y}$ be fully supported $mm$-fields over $B$. Then, $\dgp(\mm{X},\mm{Y})=0$ if and only if $\mm{X}$ and $\mm{Y}$ are isomorphic.
\end{proposition}
\begin{proof}
    The ``if'' statement is clear, so assume that $\dgp(\mm{X},\mm{Y})=0$. By Theorem \ref{thm:gp_coupling}, for each integer $n>0$, there exists a $2/n$-coupling $\mu_n$ between $\mm{X}$ and $\mm{Y}$. By Proposition \ref{prop:couplings}, by passing to a subsequence if necessary, there exists a coupling $\mu$ between $\mm{X}$ and $\mm{Y}$ such that $\mu_n$ weakly converges to $\mu$. Note that $\mu_n \otimes \mu_n$ converges weakly to $\mu \otimes \mu$ (\cite[Theorem~2.8] {billingsley2013convergence}).

    There exist Borel sets $R_n \subseteq X \times Y$ such that $\mu_n(R_n) \geq 1-1/n$ and $\dis_{\pi_X,\pi_Y}(R_n) \leq 2/n$, for all $n>0$. For an integer $N>0$, let
    \begin{equation}
        \begin{split}
            S_N:=\{(x,y,&x',y') \in X \times Y \times X \times Y \colon |d_X(x,x')-d_Y(y,y')| \leq 2/N, \\
            &d_B(\pi_X(x),\pi_Y(y)) \leq 1/N, \ \text{and} \ d_B(\pi_X(x'),\pi_Y(y')) \leq 1/N\}.
        \end{split}
    \end{equation}
    If $n \geq N$, then $R_n \times R_n \subseteq S_N$, so $\mu_n \otimes \mu_n (S_N) \geq (1-1/n)^2$. Hence, by \cite[Theorem~11.1.1]{dudley2018real}, $\mu \otimes \mu(S_N) =1$. Since $S_{N+1} \subseteq S_N$, if we let $S=\cap_N S_N$, then $\mu\otimes \mu(S)=1$ and
    \begin{equation}
    S=\{(x,y,x',y') \in X \times Y \times X \times Y: d_X(x,x')=d_Y(y,y'), \pi_X(x)=\pi_Y(y), \text{and} \, \pi_X(x')=\pi_Y(y')\}.
    \end{equation}
    Note that $\supp(\mu) \times \supp(\mu)=\supp(\mu \otimes \mu) \subseteq S$. Letting $X_0$ and $Y_0$ be the projections of $\supp(\mu)$ onto $X$ and $Y$, respectively, we have that 
    \begin{equation}
    \mu_X(\bar{X}_0)=\mu(\bar{X}_0\times Y) \geq \mu(\supp(\mu)) =1,
    \end{equation}
    where the bar denotes closure. Since $\mu_X$ is fully supported, $\bar{X}_0=X$. Similarly, $\bar{Y}_0=Y$. For each $x \in X_0$, there exists $y \in Y_0$ such that $(x,y) \in \supp(\mu)$. If $(x,y') \in \supp(\mu)$, then $(x,y,x,y') \in S$, implying that $y=y'$ and $\pi_X(x)=\pi_Y(y)=\pi_Y(y')$. Similarly, for each $y \in Y_0$, there exists unique $x$ in $X_0$ such that $(x,y) \in \supp(\mu)$. Thus, we can define a bijection $\phi: X_0 \to Y_0$ by requiring that $(x,\phi(x)) \in \supp(\mu)$. Furthermore, for all $x,x' \in X$, we have $(x,\phi(x),x',\phi(x')) \in S$, so $\pi_Y(\phi(x))=\pi_X(x)$ and $d_X(x,x')=d_Y(\phi(x),\phi(x'))$. This implies that $\phi$ is an isometry between $X_0,Y_0$, which extends to an isometry $\Phi: \mm{X} \to \mm{Y}$ of fields. It remains to show that $\phi$ is measure preserving.

    For a Borel subset $A \subseteq X$, we have $\mu_X(A)=P(A \times X) \geq P(A \times \phi(A))$. Let $A_0 = A \cap X_0$. Note that $(A \times Y) \cap \supp(P)=A_0 \times \phi(A_0)$, hence $\mu_X(A)=P(A_0 \times \phi(A_0)) \leq P(A \times \phi(A))$. Hence $\mu_X(A)=P(A \times \phi(A))$. Similarly, for a Borel subset $A$ of $Y$, we have $\mu_Y(A)=P(\phi^{-1}(A) \times A)$. Therefore,
    $$\mu_X(\phi^{-1}(A))=P(\phi^{-1}(A) \times \phi(\phi^{-1}(A)))=P(\phi^{-1}(A) \times A)=\mu_Y(A).$$
    This shows that $\phi$ is measure preserving and completes the proof.
\end{proof}

In Theorem \ref{thm:gh_metric}, we have shown that compact $B$-fields form a Polish space under the Gromov-Hausdorff distance. In contrast, the space of compact $mm$-fields over $B$ is not complete under the Gromov-Prokhorov distance since every Borel probability measure on a Polish space is the Prokhorov limit of compactly supported Borel probability measures. The next result shows that if we drop the compactness requirement, we still get a Polish space.

\begin{theorem}\label{thm:gp_metric}
    Let $\calMM_B$ denote the set of isometry classes of fully supported $B$-fields. Then, $\dgp$ metrizes $\calMM_B$ and $(\calMM_B,\dgp)$ is a Polish space.
\end{theorem}

We begin the proof with a finite approximation lemma.

\begin{lemma}\label{lem:finite_gp_approximation}
    Let $\mm{X}$ be a $B$-field, and $B_0$ be a dense subset of $B$. For each $\epsilon>0$, there exists a finite $mm$-field $\mm{Y}$ over $B$ such that $\mu_Y$ is the normalized counting measure, $\pi_Y$ takes values in $B_0$, $d_Y$ only takes rational values, and $\dgp(\mm{X},\mm{Y}) < \epsilon$.
\end{lemma}
\begin{proof}
    Varadarajan's Theorem \cite[Theorem~11.4.1]{dudley2018real} implies that there are finite subsets of $X$ whose normalized counting measures get arbitrarily $\dgp$-close to $\mu_X$. The result then follows from Lemma \ref{lem:finite_approximation}.
\end{proof}
\begin{proof}[Proof of Theorem \ref{thm:gp_metric}]
    Symmetry of $\dgp$ is clear, the triangle inequality has been proven in Proposition \ref{prop:gp_triangle} and definiteness in Proposition \ref{prop:gp_isom}. Hence, $\dgp$ is a metric on $\calMM_B$. Moreover, Lemma \ref{lem:finite_gp_approximation} shows that $\calMM_B$ has a countable dense set. Thus, it remains to show that $\calMM_B$ is complete. 
    
    Let $\{\mm{X}_n\}$ be a Cauchy sequence of $mm$-fields over $B$ with respect to Gromov-Prokhorov distance. To prove that $\{\mm{X}_n\}$ is convergent, it suffices to show that it has a convergent subsequence. By passing to a subsequence if necessary, we can assume that $\dgp(\mm{X}_n,\mm{X}_{n+1})<1/2^n$, for all $n$. For each $n>0$, there exists a coupling $\mu_n \in C(\mu_{X_n},\mu_{X_{n+1}})$ and a Borel subspace $R_n \subseteq X_n \times X_{n+1}$ such that $\mu_n(R_n) > 1- 1/2^n$ and $\dis_{\pi_{X_n},\pi_{X_{n+1}}}(R_n) <2/2^n$. Let $r_n=1/2^n$. By applying the construction described in Definition \ref{def:connecting_fields} consecutively, we get a $B$-field
    $$\mm{Z}_n:=((\mm{X}_1 \coprod_{R_1,r_1} \mm{X}_2) \coprod_{R_2,r_2} \mm{X}_3  \dots) \coprod_{R_n,r_n} \mm{X}_{n+1}$$
    for each $n>0$. By construction, $\mm{Z}_n$ is a subfield of $\mm{Z}_{n+1}$. Let $\mm{Z}$ be the completion of $\cup_{n>0} \mm{Z}_n$. Note that the union of countable dense sets in $Z_n$ is a countable dense set in $Z$, hence $Z$ is Polish. Therefore, $\mm{Z}$ is a $B$-field. By Lemma \ref{lem:connect_prokhorov}, $\dpr^Z(\mu_{X_n},\mu_{X_{n+1}}) \leq r_n = 1/2^n$. Hence, $\{\mu_{X_n}\}$ forms a Cauchy sequence with respect to Prokhorov distance in $Z$. By \cite[Corollary~11.5.5]{dudley2018real}, there exists a Borel probability measure $\mu$ of $Z$ such that $\{\mu_{X_n}\}$ Prokhorov converges to $\mu$ in $Z$. Let $Y$ be the support of $\mu$. Since $Y$ is complete, $\mm{Y}:=(Y,\mu,B,\pi_Z|_Y)$ is a fully supported $mm$-field over $B$. We have $\dgh(\mm{X}_n,\mm{Y}) \leq \dpr^Z(\mu_{X_n},\mu)$, hence $\{\mm{X}_n\}$ Gromov-Prokhorov converges to $\mm{Y}$.
\end{proof}

\begin{remark}
Analogous to Theorem \ref{thm: F_B = F_I}, there is a characterization of the Gromov-Prokhorov distance between compact $mm$-fields in terms of isometric embeddings into a Urysohn field. The statement and proof are nearly identical and therefore omitted. 
\end{remark}

\section{Gromov-Wasserstein Distance for Metric-Measure Fields}\label{sec:gw}

Given $B$-fields $(X,d_X,\pi_X)$ and $(Y,d_Y,\pi_Y)$, we let $\mXY \colon (X \times Y) \times (X \times Y) \to \R$ and $\dXY \colon X \times Y \to \R$ be the functions given by
\begin{equation}
\begin{split}
\mXY(x,y,x',y') &:=|\dX(x,x')-\dY(y,y')| \,, \\
\dXY(x,y) &:= \dB(\piX(x),\piY(y)) \,.
\end{split}
\end{equation}
Note that
\begin{equation}
    \begin{split}
        |\mXY(x_1,y_1,x_1',y_1')-\mXY(x_2,y_2,x_2',y_2')| &\leq |\dX(x_1,x_1')-\dX(x_2,x_2')| + |\dY(y_1,y_1') - \dY(y_2,y_2')| \\
        &\leq \dX(x_1,x_2) + \dX(x_1',x_2') + \dY(y_1,y_2) + \dY(y_1',y_2') \\
        &\leq 4 \max\{\dX(x_1,x_2), \dY(y_1,y_2), \dX(x_1',x_2'), \dY(y_1',y_2') \}.
\end{split}
\end{equation}
Therefore, if we endow $(X \times Y) \times (X \times Y)$ with the product sup metric, then $\mXY$ is $4$-Lipschitz. Similarly, $\dXY$ is $2$-Lipschitz. Using this notation, we introduce a field version of the Gromov-Wasserstein distance in a manner similar to \cite{memoli2011gromov, vayer2020fused}. 

\smallskip
\begin{definition}\label{def:fwd}
Let $\calX=(X,d_X,\piX,\muX)$ and $\calY=(Y,d_Y,\piY,\muY)$ be $mm$-fields over $B$. For $1 \leq p < \infty$, the {\em Gromov-Wasserstein distance} $\dgwp(\calX,\calY)$ is defined as
\begin{equation*}
\begin{split}
    \dgwp(\calX,\calY):= \inf_{\mu \in \cXY} \costpP .
\end{split}
\end{equation*}
For $p=\infty$,
\begin{equation*}
    \dgwi(\calX,\calY):= \inf_{\mu \in \cXY} \costiP.
\end{equation*}
\end{definition}

That $\dgw{p}$ and $\dgw{\infty}$ are metrics can be argued as in the case of metric measure spaces (see (\cite{memoli2011gromov}, Theorem 5.1). The arguments in Propositions \ref{prop:gwrealization} and \ref{prop:gwembedding} are similar to those in  \cite{memoli2011gromov}. However, Theorem \ref{thm:gwuniform} is new even for metric-measure spaces.

\begin{remark}\label{rem:support}
As $\mXY$ and $\dXY$ are continuous, their suprema over a set do not change after taking the closure of the set. Since the support is the smallest closest set of full measure, the suprema in the definition of $\dgwi(\calX,\calY)$ are essential suprema. 
\end{remark}

\begin{proposition}\label{prop:gwrealization}
For each $1\leq p \leq \infty$, $\dgwp(\calX,\calY)$ is realized by a coupling. Furthermore,
\begin{equation*}
\lim_{p \to \infty} \dgwp(\calX,\calY)=\dgwi(\calX,\calY).
\end{equation*}
\end{proposition}
\begin{proof}
For each integer $n\geq 1$, there exists $\mu_n \in \cXY$ such that the expression on the right-hand side in the definition of $\dgwp(\calX,\calY)$ is $ \leq \dgwp(\calX,\calY) + 1/ n$. By Proposition \ref{prop:couplings} in the Appendix, without loss of generality, we can assume that $\mu_n$ converges weakly to a coupling $\mu$. This, in turn, implies that $\mu_n \otimes \mu_n$ converges weakly to $\mu \otimes \mu$ (\cite[Theorem~2.8] {billingsley2013convergence}). We show that $\mu$ is an optimal coupling.

\smallskip
\noindent
{\bf Case 1.} Suppose that $1 \leq p < \infty$. Since $\mXY$ and $\dXY$ are continuous and bounded below, by \cite[Lemma~4.3]{villani2008optimal} we have
\begin{equation}
        \int \mXY^p \dPP \leq \liminf_n \mXY^p \dPPn 
        \quad \text{and} \quad
        \int \dXY^p \dP \leq \liminf_n \int \dXY^p \dPn \,.
\label{eq:pnp}
\end{equation} 
Using \eqref{eq:pnp} and the fact that for any sequences $\{a_n\}$ and $\{b_n\}$ of real numbers the inequality
\begin{equation}
\max \,\{\liminf_n a_n, \liminf b_n\} \leq \liminf_n \max \{a_n,b_n\} 
\end{equation}
holds, we obtain
\begin{equation}
\begin{split}
        \dgwp(\calX,\calY) &\leq \costpP \\
        & \leq \liminf_n \max \left\{ \frac{1}{2} \left( \int \mXY^p \dPPn \right)^{1/p}, \left( \int \dXY^p \dPn \right)^{1/p} \right\} \\
        &\leq \liminf_n \left( \dgwp(\calX,\calY) + 1/n \right) = \dgwp(\calX,\calY).
\end{split}
\end{equation}
This implies that $\mu$ realizes the Gromov-Wasserstein distance, as claimed.

\medskip
\noindent    
{\bf Case 2} For $p=\infty$, we adapt the proof of \cite[Proposition~3]{givens1984class} to the present setting. Note that if $1 \leq q \leq q' < \infty$, then
\begin{equation}
\dgw{q}(\calX,\calY) \leq \dgw{q'}( \calX,\calY) \leq \dgwi(\calX,\calY).
\end{equation}
Hence, we have
\begin{equation}
\lim_{q \to \infty} \dgw{q}(\calX,\calY) = \sup_q \dgw{q}(\calX,\calY) \leq \dgwi(\calX,\calY) .
\end{equation}

Let $\{q_n\}$ be a sequence of real numbers satisfying $q_n \geq 1$ and $q_n \to \infty$, and $\mu_n$ be an optimal coupling realizing $\dgw{q_n}(\calX,\calY)$. As before, we can assume that $\mu_n$ converges to $\mu$ weakly so that $\mu_n \otimes \mu_n$ converges weakly to $\mu \otimes \mu$. Let $0 \leq r < \costiP$ and set
\begin{equation}
U =\{(x,y,x',y') \colon \mXY (x,y,x',y') /2 > r \} \quad \text{and} \quad V = \{(x,y) \colon \dXY (x,y) >r \}.
\end{equation}
Either $\mu \otimes \mu \,(U)> 0$ or $\mu(V) > 0$. Let us first assume that $\mu \otimes \mu (U) = 2m > 0$. By the Portmanteau Theorem \cite[Theorem~11.1.1]{dudley2018real},
\begin{equation}
2m \leq \liminf \mu_n \otimes \mu_n(U).
\end{equation}
Passing to a subsequence if necessary, we can assume that $\mu_n \otimes \mu_n(U) \geq m> 0$, for all $n$. We then have 
\begin{equation}
        \dgw{q_n}(\calX, \calY) \geq \frac{1}{2} \bigg(\int \mXY^{q_n} \dPPn \bigg)^{1/q_n}
        \geq r \, (\mu_n \otimes \mu_n (U))^{1/q_n}  \geq r \, m^{1/q_n}.
\end{equation}
Hence,
\begin{equation}
        \lim_{p \to \infty} \dgwp(\calX,\calY) = \lim_n \dgw{q_n}(\calX,\calY) \geq r.
\end{equation}
    Since $r < \costiP$ is arbitrary, we get
    \begin{equation}
    \begin{split}
        \costiP &\geq \dgwi(\calX,\calY) \\ &\geq \lim_{p \to \infty} \dgwp(\calX,\calY) \\  &\geq \costiP.
    \end{split}
    \end{equation}
This shows that the coupling $\mu$ realizes $\dgw{\infty}(\calX,\calY)$ and also proves the convergence claim. The case $\mu(V)>0$ is handled similarly.
\end{proof}

The next proposition establishes a standard relation between the Gromov-Wasserstein and the Wasserstein distances in the setting of $mm$-fields.

\begin{proposition} \label{prop:gwembedding}
Let $\calX=(X,d_X,\pi_X,\mu_X)$ and $\calY=(Y,d_Y,\pi_Y, \mu_Y)$ be $mm$-fields over $B$. Suppose that $\piZ \colon Z \to B$ is 1-Lipschitz and $\iX \colon X \to Z$ and $\iY \colon Y \to Z$ are isometric embeddings satisfying $\piX=\piZ \circ \iota_X$ and $\piY= \piZ \circ \iY$. Then, for any $1 \leq p \leq \infty$, we have
\[
\dgwp(\calX,\calY) \leq \dwp((\iX)_*(\mu_X),(\iY)_*(\mu_Y))\,.
\]
\end{proposition}
\begin{proof}
For $1 \leq p < \infty$, let $\nu$ be an optimal coupling between $(\iX)_*(\mu_X)$ and $(\iY)_*(\mu_Y)$ realizing the $p$-Wasserstein distance $\dwp((\iX)_*(\muX), (\iY)_*(\mu_Y))$. Since $\iX (X) \times \iY (Y)$ has full measure in $(Z \times Z, \nu)$, there is a measure $\mu$ on $X \times Y$ such that $(\iX \times \iY)_*(\mu)=\nu$. Since $(\iX)_*$ and $(\iY)_*$ are injective, $\mu \in C(\muX,\muY)$. We have
\begin{equation}
\begin{split}
    \costpPm &= \bigg(\iint |\dZ(w,w')-\dZ(z,z')|^p d\nu(w,z)d\nu(w',z') \bigg)^{1/p}\\
    &\leq \bigg( \iint (\dZ(w,z)+\dZ(w',z'))^pd\nu(w,z)d\nu(w',z') \bigg)^{1/p} \\
    &\leq \bigg( \int \dZ(w,z)^pd\nu(w,z) \bigg)^{1/p} + \bigg( \int \dZ(w',z')^pd\nu(w',z') \bigg)^{1/p} \\
    &= 2 \dwp((\iX)_*(\muX), (\iY)_*(\muY)).
\end{split}
\end{equation}
Similarly,
\begin{equation}
\begin{split}
    \costpPd &= \bigg(\int \dB(\piZ(w),\piZ(z))^p d\nu(w,z) \bigg)^{1/p} \\
    &\leq \bigg(\int \dZ(w,z)^p d\nu(w,z) \bigg)^{1/p} = \dwp((\iX)_*(\muX), (\iY)_*(\muY)).
\end{split}
\end{equation}
Hence,
\begin{equation}
    \dgwp(\calX,\calY) \leq \costpP \leq \dwp((\iX)_*(\muX), (\iY)_*(\muY)).
\label{eq:lim}    
\end{equation}
Letting $p \to \infty$ in \eqref{eq:lim}, we get $\dgwi(\calX,\calY) \leq \dwi((\iX)_*(\muX), (\iY)_*(\muY))$.
\end{proof}

Recall that, in a metric measure space $(X,d_X, \mu_X)$, a sequence $(x_n)$ is called {\em uniformly distributed} (u.d.) if  $\sum_{i=1}^n \delta_{x_i}/n \to \mu_X$ weakly. Let $U_X$ denote the set of uniformly distributed sequences in $X$. It is known that $U_X$ is a Borel set in $X^\infty$ and $\mu_X^\infty(U_X)=1$ \cite{kondo2005probability}.

The next result provides a characterization of $\dgwi(\calX,\calY)$ in terms of uniformly distributed sequences.

\begin{theorem}\label{thm:gwuniform}
Let $\fmm{X} = (X, d_X, \pi_X, \mu_X)$ and $\fmm{Y} = (Y, d_Y, \pi_Y, \mu_Y)$ be bounded $mm$-fields over $B$. Then,
    $$\dgwi(\calX,\calY)= \inf_{\tiny \begin{matrix} (x_n) \in U_X \\ (y_n) \in U_Y \end{matrix}} \max \left\{ \frac{1}{2} \sup_{i,j} \mXY(x_i,y_i,x_j,y_j),
    \, \sup_i \dXY(x_i,y_i) \right\}.$$
Furthermore, there are sequences $(x_n) \in U_X$ and $(y_n) \in U_Y$ that realize the infimum on the right-hand side.
\end{theorem}
\begin{proof}
 Let us denote the infimum on the right-hand side by $\alpha$ and let $\mu$ be an optimal coupling realizing $\dgwi(\calX,\calY)$. Then,
 \begin{equation}
 \dgwi(\calX,\calY)=\costiP .
 \end{equation}
Let $(x_n,y_n)$ be an equidistributed sequence with respect to $\mu$ in $\sP$. Then, $(x_n) \in U_X$ and $(y_n) \in U_X$, and we have 
    \begin{equation}\label{eq:optimal_uniform}
    \begin{split}
        \alpha &\leq  \max \left\{ \frac{1}{2} \sup_{i,j} \mXY(x_i,y_i,x_j,y_j), \, \sup_i \dXY(x_i,y_i) \right\} \\
        &\leq \costiP = \dgwi(\calX,\calY).
    \end{split}
    \end{equation}
For the converse inequality, let $p \geq 1$, $\epsilon>0$, $(x_n) \in U_X$ and $(y_n) \in U_Y$. Let $\calE_n$ be the $mm$-field with domain $E_n= \{1,\dots,n\}$ equipped with the (pseudo)-metric $d_{E_n}(i,j)=d_X(x_i,x_j)$, normalized counting measure, and the 1-Lipschitz function $\pi_E(i)=\pi_X(x_i)$. Similarly, define $\calF_n$ using $(y_n)$. By Proposition \ref{prop:gwembedding}, we have
\begin{equation}
\dgwp(\calX, \calE_n) \leq \dwp(\mu_X, \sum_{i=1}^n \delta_{x_i}/n) \quad \text{and} \quad \dgwp(\calY, \calF_n) \leq \dwp(\mu_Y, \sum_{i=1}^n \delta_{y_i}/n) .
\end{equation}
Since the Wasserstein distances in the above inequalities converge to $0$ as $n \to \infty$ (\cite[Theorem~6.9]{villani2008optimal}), we can choose $n$ large enough so that the Gromov-Wasserstein distances in the above inequalities are $< \epsilon$. If we use diagonal coupling $\zeta_n \in C(\calE_n, \calF_n)$ given by $\zeta_n ({i,i})=1/n$, we get
\begin{equation}
\dgwp(\calE_n,\calF_n) \leq  \max \left\{ \frac{1}{2} \sup_{i,j} \mXY(x_i,y_i,x_j,y_j), \, \sup_i \dXY(x_i,y_i) \right\}.
\end{equation}
This, in turn, implies that 
 \begin{equation}
\dgwp(\calX, \calY) \leq \max \left\{ \frac{1}{2} \sup_{i,j} \mXY(x_i,y_i,x_j,y_j), \, \sup_i \dXY(x_i,y_i) \right\} + 2\epsilon.
\end{equation}
Since $(x_n) \in U_X$, $(y_n) \in U_Y$, and $\epsilon>0$ is arbitrary, we get $\dgwp(\calX,\calY) \leq \alpha$. As $p \geq 1$ is arbitrary, Proposition \ref{prop:gwrealization} implies that
\begin{equation}
\dgwi(\calX,\calY)=\lim_{p \to \infty} \dgwp(\calX,\calY) \leq \alpha .
\end{equation}
This also shows that the sequences $(x_n) \in U_X$ and $(y_n) \in U_Y$ realize the infimum.
\end{proof}

\begin{remark}
In Theorem \ref{thm:gwuniform}, setting the range space $B$ to be a singleton, we obtain a corresponding statement for metric-measure spaces.
\end{remark}
\section{Gromov-Wasserstein Through Functional Curvature} \label{sec:dmatrix}

Given a sequence $(x_n)$ is an $mm$-space $(X, d, \mu)$, one can form an associated (infinite) distance matrix $D = (d_{ij})$, where $d_{ij} = d(x_i, x_j)$. Gromov's Reconstruction Theorem \cite{gromov2007metric} states that the distribution of all distance matrices for $(X,d,\mu)$ with respect to the product measure $\mu^\infty$ is a complete invariant. This section introduces {\em augmented distance matrices} to establish a similar result for $mm$-fields and also studies relationships between the Gromov-Wasserstein distance between $mm$-fields and the Wasserstein distance between the corresponding augmented distance matrix distributions. For an integer $n>0$, let
\begin{equation}
\calR^n:=\{(r_{ij}) \in \R^{n \times n}: r_{ij}=r_{ji}, r_{ii}=0, r_{ik} \leq r_{ij}+r_{jk} \}
\end{equation}
denote the space of all $n \times n$ (pseudo) distance matrices equipped with the metric
\begin{equation}
\rho_n ((r_{ij},b_i),(r'_{ij},b'_i)):=\max \, (\frac{1}{2}\sup_{ij}|r_{ij}-r'_{ij}|,\, \sup_i d_B(b_i,b'_i)).
\end{equation}
Similarly, denoting the natural numbers by $\mathbb{N}$, let
\begin{equation}
\calR:=\{(r_{ij}) \in \R^{\mathbb{N} \times \mathbb{N}}: r_{ij}=r_{ji}, r_{ii}=0, r_{ik} \leq r_{ij}+r_{jk}  \}
\end{equation}
be the space of all countably infinite (pseudo) distance matrices equipped with the weak topology; that is, the coarsest topology that makes all projections $\pi_n \colon \calR \to \calR^n$ (onto the northwest $n \times n$ quadrant) continuous, $n >0$.
\begin{definition} \label{def matrix dstrbton}
Let $B$ be a Polish space.
\begin{enumerate}[(i)] 
\item The space of (countably infinite) {\em augmented distance matrices} (ADM) is defined as $\calR_B := \calR \times B^\infty$. 
\item Similarly, for $n> 0$, define the space of $n \times n$ {\em augmented distance matrices} as $\calR_B^n:= \fmm{R}^n \times B^n$. 
\end{enumerate}
\end{definition}
\noindent
In the study of $mm$-fields $(X, B, \pi, \mu)$, if $(x_n)$ is a sequence in $X$, we associate to $(x_n)$ the ADM defined by $r_{ij} = d(x_i, x_j)$ and $b_i = \pi (x_i)$. 

\begin{definition}[ADM Distribution]
Let $\mathcal{X}=(X, B,\pi,\mu)$ be an $mm$-field and $\calFX: X^\infty \to \mathcal{R}_B$ the map $(x_i) \mapsto (d_X(x_i,x_j),\pi(x_i))$. The {\em augmented distance matrix distribution} of $\fmm{X}$ is defined as $\calDX =(\calFX)_*(\mu^\infty)$. Similarly, for $n>0$, define $\calFXn: X^n \to \calR_B^n$ and $\mathcal{D}_\mathcal{X}^n:=(\calFXn)_*(\mu^n)$.
\end{definition}
The argument presented in \cite[p.\,122]{gromov2007metric} can be easily modified to prove the following field reconstruction result.
\begin{theorem}(Field Reconstruction Theorem) \label{thm:reconstruction}
Let $\fmm{X} = (X,d_X,\pi_X, \mu_X)$ and $\fmm{Y} = (Y,d_Y, \pi_Y, \mu_Y)$ be $mm$-fields over $B$ such that $\mu_X$ and $\mu_Y$ are fully supported. Then,
\[
\calX \simeq \calY \text{ if and only if } \calDX=\calDY.
\]
\end{theorem}


On the space $\calR_B$, we also define the (extended) metric
\begin{equation}
\rho ((r_{ij},b_i),(r'_{ij},b'_i)):=\max (\frac{1}{2}\sup_{ij}|r_{ij}-r'_{ij}|,\, \sup_i d_B(b_i,b'_i)).
\end{equation}
The metric $\rho_n$ over $\calRBn$ is defined similarly. However, since $(\calR_B, \rho)$ is not separable, instead of using $\rho$ to define a topology on $\calR_B$, we only employ it to formulate the Wasserstein distance in $\calR_B$. The next lemma shows that $\rho$ is sufficiently regular for the Wasserstein distance so defined to satisfy some desirable properties. 
\begin{lemma}\label{lemma: lower-semi-continuity of partial}
The extended function $\rho \colon \mathcal{R}_B \times \mathcal{R_B}\to [0,\infty]$ is lower semicontinuous.
\end{lemma}
\begin{proof}
Let $\pi_n: \calRB \to \calRB^n$ denote the projection map. Note that $\rho_n \circ \pi_n \uparrow \rho$ pointwise. Hence, as a pointwise supremum of a sequence of continuous functions, $\rho$ is lower semicontinuous.
\end{proof}

In the discussion below, the $p$-Wasserstein distances $\dwp (\calDX, \calDY)$ and $\dwp (\calDXn, \calDYn)$ are taken with respect to the distance functions $\rho$ and $\rho_n$, respectively.

\begin{theorem}\label{thm: dWp(Dx,Dy)=dGW(infty)}
Let $\calX$ and $\calY$ be bounded $mm$-fields over $B$. Then, for any $1 \leq p \leq \infty$, we have
\begin{equation*}\label{dgw=dwp}
    \lim_{n \to \infty} \dwp(\calDXn,\calDYn) = \dwp(\mathcal{D}_\mathcal{X},\mathcal{D}_\mathcal{Y}) = \dgwi(\mathcal{X},\mathcal{Y}) .
\end{equation*}
\end{theorem}
\begin{proof}
 The projection map $\pi_{n}^{n+1} \colon \calRB^{n+1} \to \calRBn$ is $1$-Lipschitz and has the property that
 \begin{equation}
 (\pi_{n}^{n+1})_\ast (\calDX^{n+1}) = \calDXn \quad \text{and} \quad  (\pi_{n}^{n+1})_\ast (\calDY^{n+1}) = \calDYn \,.
 \end{equation}
 This implies that $\dwp(\calDXn, \calDYn) \leq \dwp(\calDX^{n+1},\calDY^{n+1})$. By a similar argument, we get $\dwp(\calDXn,\calDYn) \leq \dwp(\calDX,\calDY)$. Therefore, we have
\begin{equation}
\begin{split}
    \lim_{n \to \infty} \dwp(\calDXn,\calDYn) = \sup_n \dwp(\calDXn,\calDYn) \leq \dwp(\calDX,\calDY).
\end{split}
\end{equation}
Since $\rho$ is lower semicontinuous by Lemma \ref{lemma: lower-semi-continuity of partial}, using an argument similar to the proof of Proposition \ref{prop:gwrealization}, one can show that $\dwi(\calDX,\calDY) = \lim_{p \to \infty} \dwp(\calDX,\calDY)$. Therefore, without loss of generality, we can assume that $p<\infty$.

Now, we show that $\dwp(\calDX,\calDY) \leq \dgwi(\calX,\calY)$. Let $\mu$ be an optimal coupling between $\mu_X$ and $\mu_Y$ realizing $\dgwi(\calX,\calY)$. Let $\psiX: (X \times Y)^\infty \to \calRB$ be the map given by $(x_n,y_n)_{n=1}^\infty \mapsto \calFX((x_n)_{n=1}^\infty)$. Define $\psiY$ similarly. Then $\nu:=(\psiX,\psiY)_*(\mu^\infty)$ is a coupling between $\calDX$ and $\calDY$. We have
\begin{equation}\label{eq:p-indep}
\begin{split}
    \dwp(\calDX,\calDY) &\leq \bigg(\int \rho^p d\nu \bigg)^{1/p} = \bigg(\int_{\sP^\infty} \rho^p(\psiX((x_n,y_n)_n),\psiY((x_n,y_n)_n) \dP^\infty ((x_n,y_n)_n) \bigg)^{1/p} \\
    &= \bigg(\int_{(\sP)^\infty} \max\big(\frac{1}{2}\sup_{i,j} \mXY(x_i,y_i,x_j,y_j), \sup_{i} \dXY(x_i,y_i)\big)^p \dP^\infty((x_n,y_n)_n) \bigg)^{1/p} \\
    &\leq \costiP = \dgwi(\calX,\calY).
\end{split}
\end{equation}
It remains to show that $\dgwi(\calX,\calY) \leq \lim_n \dwp(\calDXn,\calDYn)$. Let $0 < \epsilon < 1/2$. By Proposition \ref{prop:gwrealization}, there exists $1\leq q < \infty$ so that $\dgwi(\calX,\calY) \leq \dgw{q}(\calX,\calY) + \epsilon$. Let 
\begin{equation}
\unxqe:=\{(x_i) \in X^n \colon \dwa{q}(\muX, \sum_{i=1}^n \delta_{x_i}/n) \leq \epsilon \}
\end{equation}
and define $\unyqe$ similarly. By Proposition \ref{prop:almostuniform}, if $n$ large enough, then $\muX^n(\unxqe) \geq 1-\epsilon$ and $\muY^n(\unyqe) \geq 1 - \epsilon$. If we define $\cnxqe:=\calFXn(\unxqe)$ and $\cnyqe:=\calFXn(\unyqe)$, both of these sets are analytical, hence measurable in the completion of $\calDXn$  and $\calDYn$, respectively \cite[Theorem~13.2.6]{dudley2018real}. Moreover, the measures of these sets are $\geq 1-\epsilon$. Let $\mu_n$ be a coupling realizing $\dwp(\calDXn,\calDYn)$. Note that we have 
\begin{equation}
\begin{split}
    \mu_n(\cnxqe \times \cnyqe) \geq 1 - 2\epsilon.
\end{split}
\end{equation}
By Proposition \ref{prop:dndifference}, we also have 
\begin{equation}
    \rho_n|_{\cnxqe \times \cnyqe} \geq \dgw{q}(\calX,\calY) - 2\epsilon \geq \dgwi(\calX,\calY)-3\epsilon.
\end{equation}
Therefore,
\begin{equation}
    \dwp(\calDXn,\calDYn) = \bigg(\int \rho_n^p \,d\mu_n \bigg)^{1/p} \geq (\dgwi(\calX,\calY)-3\epsilon)(1-2\epsilon)^{1/p}.
\end{equation}
This implies that 
\begin{equation}
    \lim_{n \to \infty} \dwp(\calDXn,\calDYn) \geq (\dgwi(\calX,\calY)-3\epsilon)(1-2\epsilon)^{1/p}.
\end{equation}
Since $0 < \epsilon < 1/2$ is arbitrary, we get
\begin{equation}
    \lim_{n \to \infty} \dwp(\calDXn,\calDYn) \geq \dgwi(\calX,\calY),
\end{equation}
as claimed.
\end{proof}

\begin{corollary}
Let $\mu$ be an optimal coupling realizing $\dgwi(\calX,\calY)$ and $\nu$ be the coupling between $\calDX$ and $\calDY$ induced by $\mu$ as in the proof of Theorem \ref{thm: dWp(Dx,Dy)=dGW(infty)} . More precisely, 
$$\nu:=(\psiX,\psiY)_*(\mu^\infty),$$
where $\psiX: (X \times Y)^\infty \to \calRB$ is the map given by $(x_n,y_n)_{n=1}^\infty \mapsto \calFX((x_n)_{n=1}^\infty)$, and $\psiY$ is defined similarly.
Then, $\nu$ is the optimal coupling realizing $\dwp(\calDX,\calDY)$, independent of $p \geq 1$. Furthermore, if $(x_n,y_n)$ is a uniformly distributed sequence (with respect to $\mu$) in $\sP$, then
\begin{equation}
    \rho (\calFX((x_n)),\calFY((y_n))=\dgwi(\calX,\calY).
\end{equation}
\end{corollary}
\begin{proof}
    By \eqref{eq:p-indep}, we have
    \begin{equation}
        \dwp(\calDX,\calDY) \leq \bigg(\int \rho^p d\nu \bigg)^{1/p} \leq \dgwi(\calX,\calY),
    \end{equation}
    for all $p \geq 1$. Hence, by Theorem \ref{thm: dWp(Dx,Dy)=dGW(infty)}, the inequalities above are equalities. This shows the optimality of $\nu$ independent of $p \geq 1$.
    
    The equality $\rho (\calFX((x_n)),\calFY((y_n))=\dgwi(\calX,\calY)$ is shown in the proof of Theorem \ref{thm:gwuniform}.
\end{proof}

\begin{remark}\label{rem:discretezation}
By Theorem $\ref{thm: dWp(Dx,Dy)=dGW(infty)}$, $\dwp(\calDXn,\calDYn)$ can be used as an approximation to $\dgwi(\calX,\calY)$. To discretize this approximation, one can take i.i.d. samples from $(X^n,\mu^n)$ and $(Y^n,\nu^n)$ and form empirical measures $\calE_{n,X}$ and $\calE_{n,Y}$. Then, $\dwp(((\calFXn)_*(\calE_{n,X}), (\calFYn)_*(\calE_{n,Y}) )$ can be taken as an approximation to $\dwp(\calDXn,\calDYn)$.
\end{remark}

\begin{remark}\label{rem:p-independence}
By Theorem \ref{thm: dWp(Dx,Dy)=dGW(infty)}, $\dwp(\calDX,\calDY)$ is independent of $p \geq 1$. This can be explained as follows. Since we are using the sup-distance $\rho$ on $\calRB$ and almost every sequence in a metric measure space is uniformly distributed, if $\dwp(\calDX,\calDY) \leq r$, then there are uniformly distributed sequences in $\calX$ and $\calY$ whose augmented distance matrices are $r$-close to each other, which forces $\dwa{q}(\calDX,\calDY) \leq r$ for any $q \geq 1$.
\end{remark}

The following theorem gives a geometric representation of the isomorphism classes of $mm$-fields via the Urysohn universal field.

\begin{theorem}\label{prop:mmrespresentation}
Let $\fmm{I}_B$ denote the moduli space of isometry classes of compact $mm$-fields over $B$ endowed with the distance $\dgwi$. Let $\fmm{G}_B$ be the group of automorphisms of the Urysohn field $\fmm{U}_B$ and $\lawsB$ be the set of compactly supported laws on $\urysohnB$, endowed with the distance $\dwi$. The group $\fmm{G}_B$ acts on $\lawsB$ by $g\cdot \mu:=g_*(\mu)$. Then, $\fmm{I}_B$ is isomorphic to the orbit space of this action; that is, $\fmm{I}_B\simeq \lawsB / \fmm{G}_B$, where the orbit space is equipped with the quotient metric, as in \eqref{eq:autb}, which can be expressed as
\begin{equation*}
        d([\mu],[\nu])=\inf_{g \in \fmm{G}_B} \dwi(\mu,g_*(\nu)).
\end{equation*} 
\end{theorem}
\begin{proof}
Given a compact $mm$-field $\calX$ over $B$ and an isometric embedding $\iota: \calX \to \urysohnB$, we have $(\iota)_*(\muX) \in \lawsB$. This induces a well-defined map from $\Psi \colon \fmm{I}_B \to \lawsB / \fmm{G}_B$ because of the compact homogeneity of $\urysohnB$. 

By Proposition \ref{prop:gwembedding}, $\dgwi(\calX,\calY) \leq d(\Psi(\calX),\Psi(\calY))$.

To showthe opposite inequality, let $\mu$ be an optimal coupling realizing $\dgwi(\calX,\calY)$ and $R \subseteq X \times Y$ the support of $\mu$. If we let $r:=\dgwi(\calX,\calY)$, then $\dis_{\pi_X,\pi_Y}=2r$. Let $\calZ:=\calX \coprod_{R,r} \calY$. If we consider $\mu_X$ and $\mu_Y$ as measures on $Z$ and $\mu$ as a measure on $Z \times Z$, then $\dwi^Z(\mu_X,\mu_Y) \leq \sup_{\supp(\mu)} d_Z = \sup_R d_Z  \leq r$. If we let $\iota \colon \calZ \to \urysohnB$ be an isometric embedding, then
\begin{equation}
            d(\Psi(\calX),\Psi(\calY)) \leq \dwi(\iota_*(\muX),\iota_*(\muY)) \leq \dwi^Z(\muX,\muY) \leq \dgwi(\calX,\calY),
\end{equation}
as desired.
\end{proof}


\section{Topological Multifiltrations and Their Stability}\label{sec:filtrations}

One of the key insights of geometric and topological data analysis is that one can study the geometry of data through associated filtered spaces by analyzing the changes in topological structure along the filtration. The neighborhood (or offset) and Vietoris-Rips filtrations are two prominent examples of such constructs. Here, we investigate field analogues of these filtrations.

\subsection{Neighborhood Multifiltrations}\label{sec:nbhd}

Let $(E,d_E)$ be a metric space and $X \subseteq E$ a compact subspace. Given $r\geq 0$, the \emph{r-neighborhood} $N^r(X,E)$ of $X$ in $E$ is defined by
\begin{equation}
N^r(X,E):=\{y \in E: \dist(y,X) \leq r\}.
\end{equation}
Clearly, $N^r(X,E) \subseteq N^s(X,E)$ if $r \leq s$, so $N^*(X,E)$ is a filtration. Since $X$ is compact, denoting the closed ball in $E$ of radius $r$ centered at $x \in X$ by $B_r(x,E)$, we can  express the neighborhood filtration as follows:
\begin{equation}
        N^r(X,E) =\{y \in E: \exists x \in X \text{ such that } d_E(x,y) \leq r\} =\bigcup_{x \in X} B_r(x,E) \,.
\end{equation}

\begin{definition}[Field neighborhood bifiltration]
    Let $\calE = (E, d_E, \pi_E)$ be a $B$-field and $X \subseteq E$ a compact subspace. For $r,s \geq 0$ define 
    $$N^{r,s}(X,\calE):= \{y \in E\colon y \in N^r (X,E) \ \text{and} \ d_B(\pi_E(x),\pi_E(y)) \leq s\}.$$ 
    If $r \leq r'$ and $s \leq s'$, then $N^{r,s}(X,\calE) \subseteq N^{r',s'}(X,\calE)$ so that $N^{*,*}(X,\calE)$ forms a bifiltration that we call the \emph{field neighborhood bifiltration} of $X$ in $\calE$. 
    \end{definition}

    For $y \in E$, if we let $B_{r,s}(y,\calE):= \{y' \in E: d_E(y,y') \leq r \ \text{and} \ d_B(\pi(y),\pi(y')) \leq s\}$, then 
    \begin{equation}
    N^{r,s}(X,\calE)=\cup_{x \in X} B_{r,s}(x,\calE) .
    \end{equation}

In many applications, $X$ is a point cloud and it is desirable to have filtrations that are robust to outliers. This requires taking the density of points into account. From that point of view, instead of considering $X$ as a subspace of $E$, it is more informative to consider it as a probability measure on $E$ given by $\mu_X=\Sigma_{x \in X} \delta_x /|X|$. We can then express the field neighborhood filtration in terms of $\mu_X$ as

\begin{equation}
    \begin{split}
        N^{r,s}(X,\calE) &=\{y \in E: \exists x \in X \text{ such that } d_E(x,y) \leq r\ \text{and} \ d_B(f(x),f(y)) \leq s\} \\
                         &=\{y \in E: B_{r,s}(y,\calE) \cap X \neq \emptyset \} \\
                         &=\{y \in E: \mu_X(B_{r,s}(y,E))>0\}.
    \end{split}
\end{equation}

This last expression motivates the introduction of the following trifiltered space associated with a probability measure on the domain of a $B$-field. (This may be viewed as the $mm$-field analogue of \cite{blumberg2022stability}.) We abuse notation and also denote $mm$-fields by $\calE$.
\begin{definition}
    Let $\calE=(E, d_E, \pi_E, \mu_E)$ be an $mm$-field over $B$. For $r, s, t \geq 0$, define
    $$N^{r,s,t}(\calE):= \{y \in E: \mu_E (B_{r,s}(y,\calE)) \geq 1-t\}.$$
    If $r \leq r'$, $s \leq s'$ and $t \leq t'$, then $N^{r,s,t}(\calE) \subseteq N^{r',s',t'}(\calE)$, so that $N^{*,*,*}(\calE)$ forms a trifiltration that we refer to as the \emph{$mm$-field neighborhood trifiltration} of $\calE$.
\end{definition}

\begin{figure}[ht]
\begin{center}
\begin{tabular}{ccc}
\includegraphics[width=0.3\linewidth]{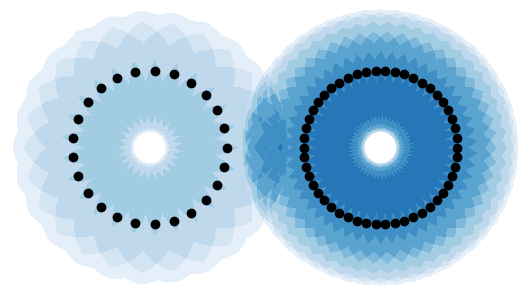} &
\includegraphics[width=0.3\linewidth]{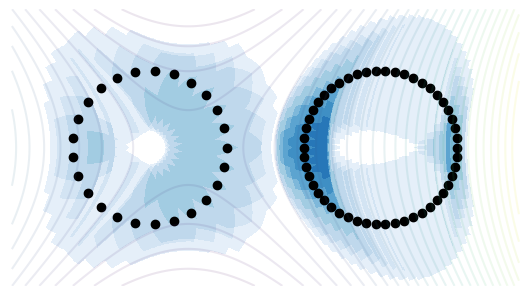} &
\includegraphics[width=0.3\linewidth]{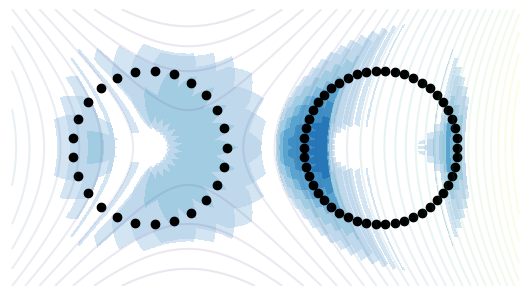} \\
(a) & (b) & (c)
\end{tabular}
\caption{Neighborhoods associated with a scalar field defined on a finite set of points $X$ sampled from two circles: (a) $N^{r}(X,E)$; (b) $N^{r,s}(X,\calE)$; (c) $N^{r,s,t}(\calE)$. The parameter values are $r=0.8$, $s=0.1$ and $t=0.99$.}
\label{fig:nbhd}
\end{center}
\end{figure}

Figure \ref{fig:nbhd} shows an example in which $X$ is a set of points sampled from two circles and $E$ is a rectangle containing $X$. A scalar field is depicted through its contour lines and $\mu_E$ is the empirical measure on $E$ induced by $X$. Panels (a), (b), and (c) compare the neighborhoods $N^{r}(X,E)$, $N^{r,s}(X,\calE)$, and $N^{r,s,t}(\calE)$, respectively, for parameter values $r=0.8$, $s=0.1$ and $t=0.99$. Note that in (c), only the more densely populated parts of the field metric neighborhood (shown in (b)) remain. In each plot, the color intensity at each point represents how densely the corresponding neighborhood of that point is populated by elements in the original point cloud.

\medskip

The main goal of this section is to establish stability properties for these multiparameter filtrations. We denote by $\dI$, $\dha$ and $\dpr$ the interleaving \cite{bubenik2015metrics}, Hausdorff \cite{burago2001course}, and Prokhorov \cite{dudley2018real} distances, respectively.

\begin{theorem}\label{thm:nbhd_stability}
    Let $\calE = (E, d_E, f, \mu)$ and $\calF = (E, d_E, g, \nu)$ be the $mm$-fields over $B$ and $X, Y \subseteq E$ be compact subspaces. Then, the following inequalities hold:
    \begin{enumerate}[{\rm (i)}]
        \item $\dI(N^{*,*}(X,\calE), N^{*,*}(Y,\calF)) \leq \dha(X,Y)+2\sup_{y \in E}d_B(f(y),g(y)).$
        \item $\dI(N^{*,*,*}(\calE), N^{*,*,*}(\calF)) \leq \dpr(\mu,\nu)+2\sup_{y \in E}d_B(f(y),g(y))$.
    \end{enumerate}
\end{theorem}
\begin{proof}
    (i) Let $\epsilon=\dha(X,Y)$ and $\epsilon'=\sup_{y \in E}d_B(f(y),g(y))$. If $z \in N^{r,s}(X,\calE)$, then there exists $x \in X$ such that $d_E(x,z) \leq r$ and $d_B(f(x),f(z)) \leq s$. Moreover, there exists $y \in Y$ such that $d_E(x,y) \leq \epsilon$. Then, $d_E(y,z) \leq r+\epsilon$ and 
    \begin{equation}
    d_B(g(y), g(z)) \leq d_B(g(y),f(y))+d_B(f(y),f(x))+d_B(f(x),f(z))+d_B(f(z),g(z)) \leq s+\epsilon+2\epsilon' .
    \end{equation}
    Therefore, $N^{r,s}(X,\calE) \subseteq N^{r+\epsilon+2\epsilon',s+\epsilon+2\epsilon'}(Y, \calF)$. Similarly, $N^{r,s}(Y,\calF) \subseteq N^{r+\epsilon+2\epsilon',s+\epsilon+2\epsilon'}(X, \calE)$. This proves the first assertion.

    \smallskip

    (ii) Let $\epsilon=\dpr(\mu,\nu)$, $\epsilon'=\sup_{y \in E}d_B(f(y),g(y))$ and note that $B_{r,s}(x,\calE) \subseteq B_{r,s+2\epsilon'}(x,\calF)$. If $y \in N^\epsilon(B_{r,s+2\epsilon'}(x,\calF))$, then there exists $z \in B_{r,s+2\epsilon'}(x,\calF)$ such that $d_E(y,z) \leq \epsilon$. Then, $d_E(x,z) \leq r+\epsilon$ and
    \begin{equation}
    d_B(g(y),g(x)) \leq d_B(g(y),g(z))+d_B(g(z),g(x)) \leq s+\epsilon+2\epsilon' .
    \end{equation}
    This shows that $N^\epsilon(B_{r,s+2\epsilon'}(x,\calF)) \subseteq B_{r+\epsilon+2\epsilon',s+\epsilon+2\epsilon'}(x,\calF)$. Therefore, if $x \in N^{r,s,t}(\calE)$, we have
    \begin{equation}
        \begin{split}
            \nu(B_{r+\epsilon+2\epsilon',s+\epsilon+2\epsilon'}(x,\calF)) &\geq \nu(N^\epsilon(B_{r,s+2\epsilon'}(x,\calF))) \\
            &\geq \mu(B_{r,s+2\epsilon'}(x,\calF))-\epsilon \\
            &\geq \mu(B_{r,s}(x,\calE))-\epsilon \geq 1-(t+\epsilon).
        \end{split}
    \end{equation}
    This shows that $N^{r,s,t}(\calE) \subseteq N^{r+\delta,s+\delta,t+\delta}(\calF)$, where $\delta=\epsilon+2\epsilon'$. Similarly, one can argue that $N^{r,s,t}(\calF) \subseteq N^{r+\delta,s+\delta,t+\delta}(\calE)$, completing the proof of (ii).
\end{proof}

\subsection{Vietoris-Rips Multifiltrations}

For a metric space $(X,d_X)$, the Vietoris-Rips simplicial filtration $\vr^*(X)$ is given by
$$\vr^r(X):=\{A \subseteq X \colon \text{$A$ is finite and $\diam(A) \leq r$}\}.$$ 
To modify this filtration for fields, recall that the \emph{radius} of a subspace $C$ of a metric space $(B,d_B)$ is defined as
$$\rad(C):=\inf_{b \in B} \sup_{c \in C} d_B(b,c).$$
It is simple to verify that $\diam(C) \leq 2\,\rad(C)$.

\begin{definition}
    Let $\calX = (X,d_X, \pi)$ be a metric field over $B$. We define the {\em metric field Vietoris-Rips bifiltration} $\vr^{\ast,\ast} (\calX)$ by
    $$\vr^{r,s}(\calX):=\{A \subseteq X \colon \text{$A$ is finite, $\diam(A) \leq r$, and $\rad(\pi(A)) \leq s/2$} \},$$
    for $r,s \geq 0$
\end{definition}

For $mm$-fields we can define a Vietoris-Rips trifiltration, as follows.

\begin{definition}
    Let $\calX = (X,d_X, \pi,\mu)$ be an $mm$-field over $B$. We define the {\em $mm$-field Vietoris-Rips trifiltration} $\vr^{\ast,\ast,\ast}(\calX)$ by
    \begin{equation*}
        \begin{split}
            \vr^{r,s,t}(\calX,\mu):=\{&\{A_0,\dots,A_n \}\colon   \text{$A_i \subseteq X$ and $A_{i-1} \subseteq A_{i}$, $\forall i$}, \\ 
            &\diam(A_i) \leq r,\, \rad(\pi(A_i)) \leq s/2,\, \mu(A_i) \geq 1-t\},
        \end{split}
    \end{equation*}
    for $r,s,t \geq 0$.
\end{definition}

    Note that if $\calX$ is a finite $mm$-field, then $\vr^{r,s,t}(\calX)$ is a full subcomplex of the barycentric subdivision of $\vr^{r,s}(\calX)$.

Figure \ref{fig:VRrst} shows an example in which $X$ is a weighted set of points sampled from a circle with the size of the dots indicating their weights. A scalar field is depicted through its contour lines. Panels (a), (b), and (c) compare the $\vr$-complex $\vr^r (X)$, the $m$-field $\vr$-complex $\vr^{r,s}(\calX)$, and the $mm$-field $\vr$-complex $\vr^{r,s,t}(\calX)$, respectively, for parameter values $r=1.5$, $s=1$ and $t=0.1$. Note how the simplices in $\vr^r(X)$ intersecting with many contour lines get removed in $\vr^{r,s}(\calX)$. As remarked above,  $\vr^{r,s,t}(\calX)$ sits in the barycentric subdivision of $\vr^{r,s}(\calX)$ and simplices away from high weight regions get removed.

\begin{figure}[ht]
\begin{center}
\begin{tabular}{ccc}
\includegraphics[width=0.2\linewidth]{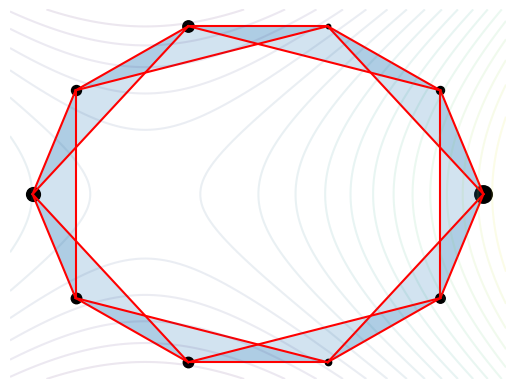} \qquad & \qquad
\includegraphics[width=0.2\linewidth]{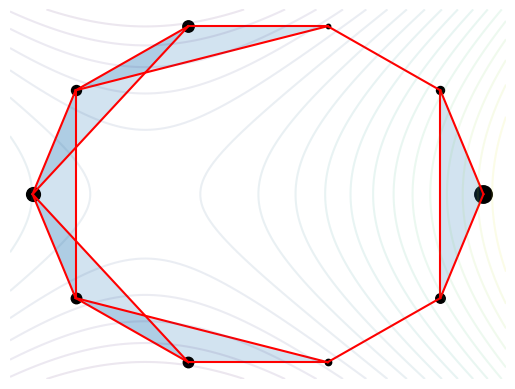} \qquad & \qquad
\includegraphics[width=0.2\linewidth]{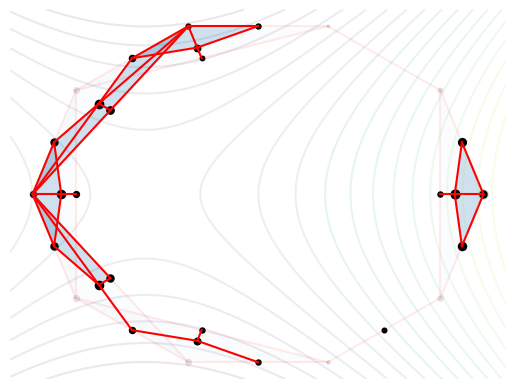}
\\
(a) \qquad & \qquad (b) \qquad & \qquad (c)
\end{tabular}
\caption{Simplicial complexes associated with a scalar field defined on a weighted finite set of points $X$: (a) the $\vr$-complex $\vr^r (X)$; (b) the $m$-field $\vr$-complex $\vr^{r,s}(\calX)$; the $mm$-field $\vr$-complex $\vr^{r,s,t}(\calX)$. The parameter values are $r=1.5$, $s=1$ and $t=0.1$.}
\label{fig:VRrst}
\end{center}
\end{figure}

We conclude with stability results for the Vietoris-Rips multiparameter filtrations defined above. The corresponding results for metric spaces and $mm$-spaces have been obtained in \cite{chazal2009gromov} and \cite{blumberg2022stability}, respectively.
\begin{theorem}
    Let $\calX$ and $\calY$ be metric measure $B$-fields. Then,
    \begin{enumerate}[{\rm (i)}]
        \item $\dHI(\vr^{*,*}(\calX), \vr^{*,*}(\calY)) \leq 2 \dgh(\calX,\calY)$ and
        \item $\dHI(\vr^{*,*,*}(\calX), \vr^{*,*,*}(\calY)) \leq 2 \dgp(\calX,\calY)$,
    \end{enumerate}
    where $\dHI$ denotes the homotopy interleaving distance \cite{blumberg2017universality}.
\end{theorem}
\begin{proof}
    (i) Let $\phi: \calX \to \calE$ and $\psi:\calY \to \calE$ be isometric embeddings. Abusing notation, we write $d_E(x,y):=d_E(\phi(x),\psi(y))$ for $x \in X$, $y \in Y$. Given $\epsilon>\dha(\phi(X),\psi(Y))$, there exist functions $f: X \to Y$ and $g: Y \to X$ such that $d_E(x,f(x)) < \epsilon$ and $d_E(g(y),y)<\epsilon$, for any $x \in X$ and $y \in Y$. Note that if $A \subseteq X$, $\diam(A) \leq r$ and $\rad(\pi_X(A)) \leq s/2$, then $\diam(f(A)) \leq r+2\epsilon$ and $\rad(\pi_Y(f(A))) \leq s/2 + \epsilon$. Hence $f$ induces a morphism $f_*:\vr^{r,s}(\calX) \to \vr^{r+2\epsilon,s+2\epsilon}(\calY)$. Similarly $g$ induces a morphism $g_*: \vr^{r,s}(\calY) \to \vr^{r+2\epsilon,s+2\epsilon}$. Let us show that $g_* \circ f_*: \vr^{r,s}(\calX) \to \vr^{r+4\epsilon,s+4\epsilon}(\calX)$ is contiguous (hence homotopy equivalent) to the inclusion map. First note that if $A$ is a simplex in $\vr^{r,s}(\calX)$, then $\diam(N^{2\epsilon}(A,X)) \leq r+4\epsilon$ and $\rad(\pi_X(N^{2\epsilon}(A,X)))\leq s/2+2\epsilon$. Since $A \cup g(f(A)) \subseteq N^{2\epsilon}(A,X)$, this implies that $A \cup g(f(A))$ is a simplex in $\vr^{r+4\epsilon,s+4\epsilon}(\calX)$. This shows the required contiguity. Similarly, $f_* \circ g_* : \vr^{r,s}(\calY) \to \vr^{r+4\epsilon,s+4\epsilon}(\calY)$ is contiguous to the inclusion. This completes the proof of part (i).
    
    (ii) Let $\epsilon>\dha(\phi_*(\mu_X),\psi_*(\mu_Y))$. Given a closed subset $A$ of $X$, let $F(A)=\psi^{-1}(N^\epsilon(\phi(A),E)) \subseteq Y$. Similarly, given a closed subset $A$ of $Y$, let $G(A)=\phi^{-1}(N^\epsilon(\psi(A),E)) \subseteq X$. Note that $\diam(F(A)) \leq \diam(A)+2\epsilon$, $\rad(F(A)) \leq \rad(A)+\epsilon$ and $\nu(F(A)) \geq \mu(A)-\epsilon$. $G$ satisfies the same inequalities. Hence, $F$ induces a simplicial map $F_*:\vr^{r,s,t}(\calX) \to \vr^{r+2\epsilon,s+2\epsilon,t+2\epsilon}(\calY)$. Similarly, $G$ induces a simplicial map. We show that $G_* \circ F_*: \vr^{r,s,t}(\calX) \to \vr^{r+4\epsilon,s+4\epsilon,t+4\epsilon}(\calX)$ is homotopy equivalent to the inclusion. We do this by constructing simplicial homotopies between both maps and the simplicial map $\Gamma: \vr^{r,s,t}(\calX) \to \vr^{r+4\epsilon,s+4\epsilon,t+4\epsilon}(\calX)$ given by $\Gamma(A):=N^{2\epsilon}(A,X)$.

    We first construct a simplicial homotopy between the inclusion and $\Gamma$. Let $I$ denote the interval simplicial complex $I:=\{\{0\}, \{1\}, \{0,1\} \}$. The simplicies of the simplicial product $\vr^{r,s,t}(\calX) \times I$ are of the form $\{(A_0,i_0),\dots, (A_n,i_n) \}$ where $A_0 \subseteq A_1 \dots \subseteq A_n \subseteq X$, $\{A_0,\dots,A_n \} \in \vr^{r,s,t}(\calX)$, $i_0, \dots i_n \in \{0,1\}$, $i_0 \leq i_1 \leq \dots i_n $. If we let $H((A,0))=A$ and $H((A,1))=\Gamma(A)$, then it gives a simplicial map $H: \vr^{r,s,t}(\calX) \times I \to  \vr^{r+4\epsilon, s+4\epsilon,t+4\epsilon}(\calX)$, which is the simplicial homotopy between the inclusion and $\Gamma$.

    Since $G(F(A)) \subseteq \Gamma(A)$, if we define $H$ instead by $H((A,0))=G(F(A))$ and $H((A,1))=\Gamma(A)$, then it becomes a simplicial homotopy between $G_* \circ F_*(A)$ and $\Gamma$. Therefore $G_* \circ F_*$ is homotopy equivalent to the inclusion. Similarly, $F_* \circ G_*$ is homotopy equivalent to the inclusion. This proves (ii).
\end{proof}

\section{Summary and Discussion} \label{sec:summary}

This paper studied functional data, termed fields, defined on geometric domains. More precisely, the objects of study were 1-Lipschitz functions between Polish metric spaces with the domain possibly equipped with a Borel probability measure. We addressed foundational questions and developed new approaches to the analysis of datasets consisting of fields not necessarily defined on the same domains. We studied the basic properties of the Gromov-Hausdorff distance between compact fields and how it relates to the Hausdorff distance in a Urysohn universal field via isometric embeddings. Similarly, we investigated analogues of the Gromov-Prokhorov and Gromov-Wasserstein distances between fields on metric-measure domains. 

We introduced a representation of metric-measure fields as probability distributions on the space of (countably) infinite augmented distance matrices and proved a Reconstruction Theorem that extends to $mm$-fields a corresponding result for $mm$-spaces due to Gromov. This provided a pathway to discrete representations of $mm$-fields via distributions of finite-dimensional augmented distance matrices for which we proved a convergence theorem. We also studied field analogues of the neighborhood and Vietoris-Rips filtrations and established stability results with respect to appropriate metrics.

Questions that also are of interest but fall beyond the scope of this paper include: (i) the study of rate of convergence of the probabilistic model based on finite-dimensional augmented distance matrices; (ii) investigation of alternative cost functions in the formulation of the Gromov-Wasserstein distance between $mm$-fields; (iii) the development of computational models and algorithms derived from augmented distance matrices.

\section*{Acknowledgements}

This work was partially supported by NSF grant DMS-1722995.


\bibliographystyle{siam}
\bibliography{realbib}

\appendix
\section{Appendix}\label{sec:appendix}


\begin{proposition}\label{prop:couplings}    
    Let $X,Y$ be Polish spaces and $\mu_X,\mu_Y$ be Borel probability measure on $X$ and $Y$ respectively. Then, every sequence in $\cXY$ has a weakly convergent subsequence.
\end{proposition}
\begin{proof}
By \cite[Theorem~9.3.7]{dudley2018real}, $\cXY$ is closed under weak convergence. By Prokhorov's Theorem \cite[Theorem~11.5.4]{dudley2018real}, it is enough to show that $C(\muX,\muY)$ is uniformly tight. Since $X$ and $Y$ are Polish, $\muX$ and $\muY$ are tight measures (\cite[Theorem~7.1.4]{dudley2018real}). Let $\epsilon>0$. There are compact subspaces $K_X \subseteq X$ and $K_Y \subseteq Y$ so that $\muX(K_X)>1-\epsilon/2$ and $\muY(K_Y)>1-\epsilon/2$. Then, for any $P$ in $\cXY$, we have 
    \begin{equation}
    \begin{split}
        P(K_X \times K_Y) &= P((K_X \times Y) \cap (X \times K_Y)) \\
        &\geq P(K_X \times Y) + P(X \times K_Y) -1 \\
        &= \muX(K_X) + \muY(K_Y)-1 \geq 1-\epsilon.
    \end{split}
    \end{equation}
    This concludes the proof.
\end{proof}

\begin{proposition}\label{prop:almostuniform}
Let $1\leq p < \infty$ and $(X,\dX,\muX)$ be a metric measure space such that $\mu_X$ has finite moments of order $p$, where $1 \leq p < \infty$. Given an integer  $n>0$ and $\epsilon>0$, let 
\begin{equation*}
\unxpe:=\{(x_i) \in X^n \colon \dwp(\muX, \sum_{i=1}^n \delta_{x_i}/n) \leq \epsilon \} .
\end{equation*}
Then, for $n$ sufficiently large, we have
 \begin{equation*}
 \mu^N(\unxpe) \geq 1 - \epsilon.
 \end{equation*}
\end{proposition}
\begin{proof}
 Let $P_p(X)$ denote the set of Borel probability measures on $X$ with finite moments of order $p$. $P_p(X)$  is metrizable by $\dwp$, and the corresponding notion of convergence is weak convergence \cite[Theorem~6.9]{villani2008optimal}. Furthermore, $P_p(X)$ is complete and separable \cite[Theorem~6.18]{villani2008optimal}. Let $\pi_n \colon X^\infty \to X^n$ denote the projection onto the first $n$ coordinates and $\psi_n \colon X^\infty \to P_p(X)$ be the map given by
 \begin{equation}
\psi_n((x_i)):= \sum_{i=1}^n \delta_{x_i}/n.
 \end{equation}
 This is a continous map. By Varadarajan Theorem \cite[Theorem~11.4.1]{dudley2018real}, $(\psi_n)$ converges to $\muX$ almost surely. By \cite[Theorem~9.2.1]{dudley2018real}, $(\psi_n)$ converges to $\mu$ in probability. Hence, for $n$ large enough, 
\begin{equation}
1 - \epsilon \leq \mu^\infty(\dwp(\muX,\psi_n) \leq \epsilon) = \mu^\infty(\pi_n^{-1}(\unxpe))=\mu^n(\unxpe),
\end{equation}
as desired.
\end{proof}

\begin{proposition}\label{prop:dndifference}
Let $\calX$ and $\calY$ be $mm$-fields over $B$. For $\epsilon >0$, let $\unxpe$ and $\unype$ be defined as in Proposition \ref{prop:almostuniform}. If $(x_i) \in \unxpe$, and $(y_i) \in \unype$, then
 \begin{equation*}
 \rho_n (\calFXn((x_n)), \calFYn((y_n))) \geq \dgwp(\calX,\calY) - 2\epsilon.
 \end{equation*}
\end{proposition}
\begin{proof}
Let $\calE_X$ be the $mm$-field with underlying set $\{1,\dots,n \}$ equipped with the (pseudo)-metric given by $d_E(i,j)=\dX(x_i,x_j)$, normalized counting measure, and $1$-Lipschitz function $i \mapsto \piX(x_i)$. Similarly, define $\calE_Y$ using $(y_n)$. Note that, by Proposition \ref{prop:gwembedding}, $\dgwp(\calE_X,\calX) \leq \epsilon$ and $\dgwp(\calE_Y,\calY) \leq \epsilon$. If $P$ is the diagonal coupling between the measures of $\calE_X$ and $\calE_Y$, then we have
\begin{equation}
\begin{split}
        \rho_n (\calFXn((x_i)), \calFYn((y_i))) &= \max \left\{  \frac{1}{2} \sup_{\sPP} m_{E_X,E_Y}, \sup_{\sP} d_{E_X,E_Y} \right\} \\
        &\geq \dgwp(\calE_X,\calE_Y) \geq \dgwp(\calX,\calY) - 2\epsilon\,,
\end{split}
\end{equation}  
as claimed.
\end{proof}

\end{document}